\documentclass[11pt,reqno]{amsart}

\usepackage{mathrsfs}
\usepackage{amsfonts}
\usepackage{amssymb}
\usepackage{amsxtra}
\usepackage{dsfont}
\usepackage[dvipsnames]{xcolor}
\usepackage[english]{babel}
\usepackage{enumerate}
\usepackage[shortlabels]{enumitem}
\usepackage{todonotes}

\usepackage{url}
\usepackage{nicematrix,tikz}

\usepackage[compress, sort]{cite}

\usepackage[T1]{fontenc}

\usepackage{newtxmath}
\usepackage[type1]{newtxtext}

\usepackage[margin=2.5cm, centering]{geometry}
\usepackage[colorlinks,citecolor=blue,urlcolor=blue,bookmarks=true]{hyperref}

\usepackage{textcomp}
\usepackage{bm}

\hypersetup{
pdfpagemode=UseNone,
pdfstartview=FitH,
pdfdisplaydoctitle=true,
pdfborder={0 0 0}, 
pdftitle={On positive Jacobi matrices with compact inverses},
pdfauthor={Pavel Šťovíček and Grzegorz Świderski},
pdflang=en-GB
}

\newcommand{\CC}{\mathbb{C}}

\newcommand{\NN}{\mathbb{N}}
\newcommand{\RR}{\mathbb{R}}

\newcommand{\calC}{\mathcal{C}}

\newcommand{\calS}{\mathcal{S}}

\newcommand{\calG}{\mathcal{G}}

\newcommand{\calJ}{\mathcal{J}}
\newcommand{\calZ}{\mathcal{Z}}

\newcommand{\frakF}{\mathfrak{F}}
\newcommand{\frakp}{\mathfrak{p}}
\newcommand{\frake}{\mathfrak{e}}
\newcommand{\fraks}{\mathfrak{s}}

\newcommand{\frakW}{\mathfrak{W}}

\newcommand{\Id}{I}

\newcommand{\norm}[1]{\lVert {#1} \rVert}

\newcommand{\tr}{\operatorname{tr}}

\DeclareMathOperator{\lin}{span}

\newcommand{\Dom}{\operatorname{Dom}}

\newcommand{\sigmaEss}{\sigma_{\mathrm{ess}}}
\newcommand{\sigmaP}{\sigma_{\mathrm{p}}}

\newcommand{\ud}{{\: \rm d}}

\newcommand{\supp}{\operatornamewithlimits{supp}}

\newcommand{\rphis}[5]{\,_{#1}\phi_{#2} \left( \genfrac{.}{.}{0pt}{}{#3}{#4}
\ ; #5 \right)}

\renewcommand{\lin}{\mathop\mathrm{span}\nolimits}
\newcommand{\sgn}{\mathop\mathrm{sgn}\nolimits}
\newcommand{\Tr}{\mathop\mathrm{Tr}\nolimits}
\newcommand{\Res}{\mathop\mathrm{Res}}

\newtheorem{theorem}{Theorem}[section]

\newtheorem{proposition}[theorem]{Proposition}
\newtheorem{lemma}[theorem]{Lemma}
\newtheorem{corollary}[theorem]{Corollary}

\theoremstyle{plain}
\newcounter{thm}

\numberwithin{equation}{section}

\theoremstyle{definition}
\newtheorem{example}[theorem]{Example}
\newtheorem{remark}[theorem]{Remark}

\newtheorem*{remark*}{Remark}

\title{On positive Jacobi matrices with compact inverses}

\author{Pavel \v{S}\v{t}ov\'{\i}\v{c}ek}
\address{
	Pavel \v{S}\v{t}ov\'{\i}\v{c}ek\\
        Department of Mathematics\\
        Faculty of Nuclear Science\\
        Czech Technical University in Prague\\
        Trojanova 13\\
        120 00 Praha, Czech Republic
}
\email{stovicek@fjfi.cvut.cz}

\author{Grzegorz \'{S}widerski}
\address{
    Grzegorz \'{S}widerski \\
    Institute of Mathematics \\ 
    Polish Academy of Sciences \\
    ul. Śniadeckich 8 \\
    00-696 Warsaw, Poland
}
\email{grzegorz.swiderski@pwr.edu.pl}
\curraddr{Faculty of Pure and Applied Mathematics, Wroclaw University of Science and Technology, Wyb. Wyspiańskiego 27, 50-370 Wroclaw, Poland}


\begin{document}
\selectlanguage{english}

\begin{abstract}
We consider positive Jacobi matrices $J$ with compact inverses and consequently
with purely discrete spectra.
A number of properties of the corresponding sequence of orthogonal polynomials is studied including the convergence of their zeros, the vague convergence
of the zero counting measures and of the Christoffel--Darboux kernels on the diagonal.
Particularly, if the inverse of $J$ belongs to some Schatten class,
we identify the asymptotic behaviour of the sequence of orthogonal polynomials and express it in terms of its regularized characteristic function.
In the even more special case when the inverse of $J$ belongs to the trace class 
we derive various formulas for the orthogonality measure, eigenvectors of $J$ as well as for the functions of the second kind and related objects.
These general results are given a more explicit form in the case when $-J$ is a generator of a Birth--Death process.
Among others we provide a formula for the trace of the inverse of $J$.
We illustrate our results by introducing and studying a modification of $q$-Laguerre polynomials corresponding to a determinate moment problem.
\end{abstract}

\maketitle

\section{Introduction} \label{sec:1}
Given two sequences $a=(a_n : n \in \NN_0), b=(b_n : n \in \NN_0)$ of positive and real numbers, respectively, we define an infinite tridiagonal matrix by
\[
	\calJ =
    \begin{pmatrix}
    b_0	& a_0	&		&       \\
	a_0 & b_1	& a_1	&         \\
		& a_1	&  b_2 	& \ddots \\
	&		&		\ddots & \ddots  
    \end{pmatrix}.
\]
The action of $\calJ$ is well-defined on the linear space $\ell(\NN_0)$ of all complex-valued sequences treated as column vectors. Let the operator $J$ be the closure in the Hilbert space $\ell^2(\NN_0)$ of the action of $\calJ$ on the sequences with finite support. Let us recall that
\[
	\ell^2(\NN_0) = \bigg\{ x \in \ell(\NN_0) : \sum_{n=0}^\infty |x_n|^2 < \infty \bigg\}, \quad
	\langle x,y \rangle = \sum_{n=0}^\infty \overline{x_n} y_n.
\]
For any $n \in \NN_0$ by $\bm{e}_n$ we shall denote the sequence with one on the $n$th position and zero elsewhere. Then $(\bm{e}_n : n \in \NN_0)$ is an orthonormal basis of $\ell^2(\NN_0)$, the so called standard (or canonical) orthonormal basis.

The symmetric operator $J$ always has self-adjoint extensions on $\ell^2(\NN_0)$. Let $\tilde{J}$ be any such extension. 
Define a probability measure on the Borel subsets of the real line by
\begin{equation}
	\label{eq:measureOrth}
	\mu(\cdot) = \langle \bm{e}_0, E_{\tilde{J}}(\cdot) \bm{e}_0 \rangle,
\end{equation}
where $E_{\tilde{J}}$ is the projection-valued spectral measure of $\tilde{J}$, i.e. $\tilde{J} = \int_\RR x \ud E_{\tilde{J}}(x)$ in the weak sense. The support of $\mu$ coincides with the spectrum $\sigma(\tilde{J})$.

Associated with $\calJ$ is the sequence of polynomials $(P_n : n \in \NN_0)$ satisfying the three-term recurrence relation
\begin{equation} \label{eq:1}
	z P_n(z) = a_{n-1} P_{n-1}(z) + b_n P_n(z) + a_n P_{n+1}(z), \quad n \geq 1
\end{equation}
with the initial conditions
\begin{equation} \label{eq:2}
	P_0(z) \equiv 1, \quad P_1(z) = \frac{z-b_0}{a_0}.
\end{equation}
It turns out that $(P_n : n \geq 0)$ is a sequence of orthonormal polynomials in $L^2(\mu)$, i.e.
\[
	\int_\RR \overline{P_n(x)} P_m(x) \ud \mu(x) =
	\begin{cases}
		1 & n=m, \\
		0 & \text{otherwise}.
	\end{cases}
\]
It is well-known that the Hamburger moment problem for such an orthonormal polynomial sequence $(P_n(x) : n \geq 0)$
is determinate, i.e. the sequence is orthonormal with respect to a unique probability measure on the real line, if and only if the Jacobi operator $J$ itself is self-adjoint (see e.g. \cite[Theorem 6.10]{Schmudgen2017}).

For our purposes particularly interesting will be those Jacobi matrices which are related to generators of Birth-Death processes. In that case
\begin{equation} \label{eq:34}
	a_n = \sqrt{\lambda_n \mu_{n+1}}, \quad
	b_n = \lambda_n + \mu_n,
\end{equation}
where $(\lambda_n : n \in \NN_0)$ and $(\mu_n : n \in \NN_0)$ are sequences such that $\lambda_n > 0, \mu_{n+1} > 0$ for any $n \in \NN_0$, and $\mu_0 \geq 0$.
For the motivation of Jacobi parameters \eqref{eq:34} see e.g. \cite[Section 5.1]{Monotonic}. For probabilistic applications of $-J$ see e.g. \cite[Chapter 3]{Iglesia2021}.

In this paper we focus mainly on the case when $J$ is positive-definite and the inverse of its specific self-adjoint extension, the so-called \emph{Friedrichs extension}, belongs to the trace class or, more generally, to a Schatten class. Some first results on this topic have been obtained in papers \cite{Stovicek2020,Stovicek2022}. Here these results are substantially extended and generalized in several directions.

In Sections \ref{sec:5} and \ref{sec:4} we assume that $J$ is positive-definite
and $J^{-1}$ belongs to the trace class.
This situation is firstly studied, in Section~\ref{sec:5}, on a general level.
Naturally, under the aforementioned assumptions the spectrum of $J$ is pure point and also the orthogonality measure for $(P_n(x) : n \geq 0)$ is purely atomic. As a first step, we derive a convenient
criterion, expressed in terms of the real sequence $(P_n(0) : n \geq 0)$, which implies the desired properties of $J$.
Then we derive several formulas for the orthogonality measure and for eigenvectors of $J$.
The characteristic function of $J$ and its associated semi-infinite principal submatrices as well as the functions of the second kind play an important role in this study.

Then, in Section~\ref{sec:4}, we specialize these general results to Jacobi operators corresponding
to generators of Birth-Death processes. This is a broad class of Jacobi operators which are widely used in various
applications. For this type of Jacobi operators we derive a criterion for the trace class inverse,
a formula for the characteristic function and an additional formula related to the orthogonality measure. 
This way we extend substantially the results obtained earlier in~\cite{Stovicek2022}
but only for a very special class of Jacobi operators.

In Section~\ref{sec:4new} we advance our studies by considering a more general case
when the inverse of the Friedrichs extension of $J$ is supposed to be compact.
In Section~\ref{sec:zero} we study weak convergence of the zero counting measures of $P_n$ and in Section~\ref{sec:CDKernel} the weak behaviour of the  Christoffel--Darboux kernels on the diagonal.
In Section~\ref{sec:2}, under the additional assumption that the inverse of the Friedrichs extension of $J$ belongs to an arbitrary Schatten class, the regularized characteristic function of the Jacobi operator is defined and it is shown to be the locally uniform limit
of a certain sequence of elementary analytic functions.
Thus we generalize a result from \cite{Stovicek2020} where
the characteristic function of $J$ has been introduced and studied
under the assumption that $J^{-1}$ is a trace class operator.
Moreover, in Section~\ref{sec:pos} we show that for Jacobi operators which are not semibounded a similar convergence does not hold in general. 

Finally as an application, in Section~\ref{sec:6},
we introduce and study an alternative $q$-version to the Laguerre polynomials
$(L_{n}^{(\alpha)}(x) : n \geq 0)$ which we call the modified $q$-Laguerre polynomials
and denote $(\tilde{L}_{n}^{(\alpha)}(x;q) : n \geq 0)$.
The particular case with $\alpha=0$ has been studied in~\cite{Stovicek2022}.
A distinguished feature of these newly introduced $q$-Laguerre polynomials is that
the respective Hamburger moment problem is determinate.
Furthermore, they can be shown to correspond to the generators of a Birth-Death process.
They are related to the (standard) $q$-Laguerre polynomials, as introduced and studied in~\cite{Moak1981},
by a construction based on the notion of quasi-orthogonal sequence~\cite{Bracciali2018}.
The $q$-Laguerre polynomials are known explicitly and, owing to this kind of construction,
the same is in principle true for the modified $q$-Laguerre polynomials.
Notably, the moment sequence for $(\tilde{L}_{n}^{(\alpha)}(x;q) : n \geq 0)$ can be derived explicitly
and the orthogonality measure turns out to be supported on the roots
of the (appropriately rescaled) $q$-Bessel function $J_{\alpha+1}^{(2)}(2\sqrt{z};q)$.

\section{More on positive Jacobi operators with a trace class inverse} \label{sec:5}

\subsection{A criterion for the trace class}
In Sections \ref{sec:5} and \ref{sec:4} we assume that $J$ is positive-definite and $J^{-1}$ is
a trace class operator. This is what has been assumed also in the earlier paper \cite{Stovicek2020}.
In this subsection we prove a convenient criterion that guarantees $J^{-1}$ to exist and to belong to the trace class. It turns out that with this criterion the assumption on positive-definiteness of $J$ can be relaxed so that it is fully sufficient to assume that $J$ is just positive-semidefinite meaning that its quadratic form is non-negative. The positive-definiteness is then obtained as a straightforward consequence.
This simplifies the discussion notably.

In the sequel, the orthogonality measure for the orthonormal polynomial sequence
$(P_{n}(x) : n \geq 0)$ is called $\mu$ and $\varrho(J)$ stands for
the resolvent set of $J$ while $\sigma(J)$ denotes its spectrum. 

\begin{theorem} \label{thm:T1} 
Assume that
\begin{equation}
	\label{eq:sum-Pn2}
	\sum_{n=0}^{\infty} P_{n}(0)^{2} = \infty
\end{equation}
and $J$ is positive-semidefinite. Then
$J$ is self-adjoint and invertible. Moreover, $J^{-1}$ belongs to
the trace class if and only if
\begin{equation}
	\label{eq:sum-trace}
	-\sum_{n=0}^{\infty}
	\Bigg(
	\sum_{j=n}^{\infty}
	\frac{1}{a_{j} P_{j}(0) P_{j+1}(0)}
	\Bigg)
	P_{n}(0)^{2} < \infty.
\end{equation}
If so then $\Tr(J^{-1})$ equals the sum in \eqref{eq:sum-trace} (including the sign).
\end{theorem}

\begin{proof} 
It follows from \eqref{eq:sum-Pn2} that the Hamburger
moment problem for $J$ is determinate and so $J$ is self-adjoint,
and also that $0 \notin \sigmaP(J)$, i.e. $J^{-1}$ exists.
The positivity of $J$ then implies that the support of the spectral measure of $J$ is contained
in the closed positive real half-line, and the same is true for the orthogonality measure $\mu$.
Then all roots of the polynomials
$P_{n}(x)$, $n \geq 1$, are known to be contained in the interior of the supporting set for the moment functional which means that they must be contained in $(0,+\infty)$
(see e.g. \cite[Theorem I.5.2]{Chihara1978}).

Moreover, recalling that $(a_{n} : n \geq 0)$ is a sequence of positive numbers, we have $\sgn P_{n}(0)=(-1)^{n}$ for any $n \geq 0$, and for any $n \geq 1$ the functions $P_{n}(x)^{2}$ are all decreasing on the
interval $(-\infty, 0]$.

According to equations (28) and (7) in \cite{Stovicek2020}, for
all $z \in \varrho(J)$,
\begin{align*}
	-\Bigg(
	\sum_{j=n}^{\infty}
	\frac{1}{a_{j} P_{j}(z) P_{j+1}(z)}
	\Bigg)
	P_{n}(z)^{2}
	&=
	\langle \bm{e}_{n}, (J - z \Id)^{-1} \bm{e}_{n} \rangle \\
	&=
	\int_\RR 
	\frac{P_{n}(x)^{2}}{x-z} \ud \mu(x).
\end{align*}
Then, since~$\mu$ is supported on the positive
half-line, we can apply twice the monotone convergence theorem and
get
\begin{align*}
	\int_\RR
	\frac{P_{n}(x)^{2}}{x} \ud \mu(x) 
	&=
	\lim_{\varepsilon \to 0+}
	\int_\RR
	\frac{P_{n}(x)^{2}}{x+\varepsilon} \ud \mu(x) \\
	&= 
	-\lim_{\varepsilon \to 0+} 
	\Bigg(
	\sum_{j=n}^{\infty}\frac{1}{a_{j} P_{j}(-\varepsilon) P_{j+1}(-\varepsilon)}
	\Bigg)
	P_{n}(-\varepsilon)^{2} \\
	&=
	-\Bigg(
	\sum_{j=n}^{\infty}
	\frac{1}{a_{j} P_{j}(0) P_{j+1}(0)}
	\Bigg) 
	P_{n}(0)^{2}.
\end{align*}
Thus, for any given $n \in \NN_{0}$, $\bm{e}_{n} \in \Dom(J^{-1/2})$
if and only if
\begin{equation}
	\label{eq:J-sqrt-en}
	\|J^{-1/2} \bm{e}_{n}\|^{2}
	=
 \int\frac{P_{n}(x)^{2}}{x}\,\text{d}\mu(x)
    =
	-\Bigg(
	\sum_{j=n}^{\infty}
	\frac{1}{a_{j} P_{j}(0) P_{j+1}(0)}
	\Bigg)
	P_{n}(0)^{2} < \infty.
\end{equation}

If $J^{-1}$ belongs to the trace class then $J^{-1/2}$ is a Hilbert-Schmidt
operator and, in regard of \eqref{eq:J-sqrt-en}, condition \eqref{eq:sum-trace} holds.

Conversely, if \eqref{eq:sum-trace} is true then $\forall n \in \NN_{0}$,
$\bm{e}_{n} \in \Dom(J^{-1/2})$, and, by a standard estimate, $J^{-1/2}$
is bounded on $\lin \{\bm{e}_{n} : n \geq 0\}$. In detail, $\forall \bm{v} \in \lin \{\bm{e}_{n} : n \geq 0\}$,
\[
	\|J^{-1/2} \bm{v}\|^{2}
	=
	\sum_{n=0}^{\infty}
	|\langle \bm{e}_{n}, J^{-1/2} \bm{v} \rangle|^{2}
	=
	\sum_{n=0}^{\infty}
	|\langle J^{-1/2} \bm{e}_{n},\bm{v} \rangle|^{2}
	\leq
	\sum_{n=0}^{\infty}
	\|J^{-1/2}\bm{e}_{n}\|^{2} 
	\|\bm{v}\|^2.
\]
Consequently, since $J^{-1/2}$ is closed, it is necessarily an everywhere
defined bounded operator. And again by \eqref{eq:sum-trace}, $J^{-1/2}$
is a Hilbert-Schmidt operator, $J^{-1}$ is a trace class operator
and
\[
	\Tr(J^{-1})
	=
	\sum_{n=0}^{\infty}
	\langle \bm{e}_{n}, J^{-1} \bm{e}_{n} \rangle
	=
	\sum_{n=0}^{\infty} 
	\|J^{-1/2}\bm{e}_{n}\|^{2}
	=
	-\sum_{n=0}^{\infty}
	\Bigg(
	\sum_{j=n}^{\infty}
	\frac{1}{a_{j} P_{j}(0) P_{j+1}(0)}
	\Bigg)
	P_{n}(0)^{2}.
\]
This concludes the proof. 
\end{proof}

\subsection{Functions of the second kind, the orthogonality measure} \label{sec:5:1}

Using the same notation as in Section~\ref{sec:1}, here we
fully focus on the case when the Jacobi operator $J$ is positive-definite
and its inverse belongs to the trace class. This also means that $0$
belongs to the resolvent set of $J$. Moreover, then the spectrum
of $J$ is discrete and consists of simple eigenvalues of $J$ only.
The eigenvalues can be enumerated,
\begin{equation*}
    \sigma(J) = \{ \zeta_{n} : n \geq 1 \},
\end{equation*}
so that $(\zeta_{n})$ is a strictly increasing sequence of positive
numbers, and $\lim_{n \to \infty} \zeta_{n} = +\infty$.

\emph{Weyl's function} $w(z)$ associated with $J$ is by definition
the Stieltjes transformation of the measure $\mu$,
\[
	w(z)
	=
	\int_\RR 
	\frac{\text{d}\mu(x)}{x-z}
	=
	\langle \bm{e}_{0}, (J-z \Id)^{-1} \bm{e}_{0} \rangle.
\]
More generally, the so called \emph{functions of the second kind}
are defined as
\begin{equation}
	\label{eq:fce-second-kind}
	w_{k}(z) 
	= 
	\int_\RR 
	\frac{P_{k}(x)}{x-z} \ud \mu(x)
	=
	\langle \bm{e}_{k},(J-z\Id)^{-1} \bm{e}_{0} \rangle, \quad k \geq 0.
\end{equation}
In particular, $w(z) \equiv w_{0}(z)$. In our case, $w_{k}(z)$ are
meromorphic functions on $\CC$, its singular points are located
at the eigenvalues of $J$, and all of them are simple poles.
The measure $\mu$ can be reconstructed from Weyl's function by evaluating
the residues,
\[
	\forall \zeta \in \sigma(J),\ 
	\Res_{z = \zeta} w(z) = -\mu(\{\zeta\}).
\]

The characteristic function of $J$ can be defined naturally as the Fredholm determinant $\det (\Id - z J^{-1})$, which results in the formula
\[
	\frakF(z)
	=
	\prod_{n=1}^{\infty} 
	\bigg( 1-\frac{z}{\zeta_{n}} \bigg),
\]
A more general situation will be considered in Subsection~\ref{sec:2}
and this definition is a particular case of~\eqref{eq:8} for $k=1$ (see also \eqref{eq:62} for the definition in terms of the regularized Fredholm determinant).
$\frakF(z)$ is an entire function whose zeroes are all simple
and coincide with the eigenvalues of $J$. It is known that 
\begin{equation}
	\label{eq:Fchar-sum-wP}
	\frakF(z) 
	= 
	1 - z \sum_{n=0}^{\infty} w_{n}(0) P_{n}(z),
\end{equation}
and the series converges locally uniformly in $z \in \CC$, see \cite[Thm. 2]{Stovicek2020}.

Let
\begin{equation}
	\label{eq:def-W}
	\frakW(z) 
	= 
	w(z) \frakF(z).
\end{equation}
Then, clearly, all singularities of $\frakW(z)$ are removable
which is to say that $\frakW(z)$ is an entire function. For every $\zeta \in \sigma(J)$ we have $\frakF(\zeta)=0$, and thus
\[
	\mu(\{\zeta\}) 
	= 
	-\lim_{z \to \zeta} 
	(z - \zeta) w(z)
	=
	-\lim_{z \to \zeta}
	\frac{(z - \zeta) \frakW(z)}{\frakF(z) - \frakF(\zeta)}.
\]
Hence
\begin{equation} 
	\label{mu-W-F}
	\forall \zeta \in \sigma(J),\ 
	\mu(\{\zeta\}) = -\frac{\frakW(\zeta)}{\frakF'(\zeta)}.
\end{equation}

It may be worthwhile to recall, too, that
\[
	\forall \zeta \in \sigma(J),\ 
	\mu(\{\zeta\}) = 
	\Bigg( 
	\sum_{n=0}^{\infty} P_{n}(\zeta)^2
	\Bigg)^{-1},
\]
see e.g. \cite[Corollary 2.6, p. 45]{Shohat1943} or \cite[Thm. 2.5.3]{Akhiezer1965}.

As an auxiliary step in the course of the proof of Theorem 2 in \cite[p.\,11]{Stovicek2020},
a formula for $w_n(0)P_n(z)$ was derived but not stated there as an independent
proposition (the third displayed equation on the indicated page, not numbered).
Here this result is recalled as a lemma.

\begin{lemma}\label{thm:Fchar-aux} 
For $n \geq 0$,
\[
	w_{n}(0) P_{n}(z)
	=
	\sqrt{\kappa_{n}}
	\langle \bm{e}_{n}, (I-zK)^{-1} \bm{k} \rangle,
\]
where
\begin{equation}
\label{eq:bm-k}
	\kappa_{n} = 
	w_{n}(0) P_{n}(0),\quad
	\bm{k} = 	
 (\sqrt{\kappa_{0}},\sqrt{\kappa_{1}},\sqrt{\kappa_{2}},\ldots)^{T},
\end{equation}
and
\begin{equation} 
	\label{eq:K}
	K_{m,n} 
	= 
	\begin{cases}
		\displaystyle{\frac{P_{n}(0)}{P_{m}(0)}
		\sqrt{\frac{\kappa_{m}}{\kappa_{n}}}
		\big( w_{m}(0) P_{n}(0) - w_{n}(0) P_{m}(0) \big)} & \text{for } m > n\\
	\noalign{\medskip}	
        0 & \text{otherwise}.
	\end{cases}
\end{equation}
Hence, in view of \eqref{eq:Fchar-sum-wP}, the characteristic function of $J$ reads
\begin{equation}
	\label{eq:Fchar}
	\frakF(z) = 1 - z \mathbf{k}^{T} \cdot(\Id-zK)^{-1} \cdot \mathbf{k}.
\end{equation}
\end{lemma}

\begin{remark*} 
The inverse operator $(\Id-zK)^{-1}$ can be evaluated
as a geometric series, and doing so the scalar product $\langle \bm{e}_{n}, (\Id-zK)^{-1} \bm{k} \rangle$
reduces to a finite sum. 
\end{remark*}

Further we recall that the sequence of so-called \emph{orthogonal polynomials
of the second kind}, $(Q_{n} : n \geq 0)$, is defined by the three-term
recurrence relation
\begin{gather*}
	Q_{0}(x) = 0, \quad Q_{1}(x) = \frac{1}{a_0}, \\
	a_{n-1} Q_{n-1}(x) + b_{n} Q_{n}(x) + a_{n} Q_{n+1}(x) = x Q_{n}(x), \quad n \geq 1,
\end{gather*}
(see e.g. \cite[Chapter\,1, \S\,2.1]{Akhiezer1965} or
\cite[Section\,5.4]{Schmudgen2017})(\footnote{Some authors can use a different terminology, these polynomials in the monic form are called the monic numerator polynomials in \cite[Section\,III.4]{Chihara1978}.})
and it holds true that
\[
	Q_{n}(x) = 
	\int_{\RR}
	\frac{P_{n}(x)-P_{n}(u)}{x-u} \ud \mu(u), \quad n \geq 0.
\]

In the beginning of the proof of Theorem~2 in \cite{Stovicek2020}, 
it was shown that
\begin{equation}
	\label{eq:lim-F}
	\frakF(z) 
	= 
	\lim_{n \to \infty}
	\frac{P_{n}(z)}{P_{n}(0)}
\end{equation}
locally uniformly on $\CC$. For a different proof, see~Corollary~\ref{cor:1} below.

In \cite{Berg1994}, a generalization of Markov's Theorem has been derived
which is applicable to any Jacobi matrix $\calJ$ which is the
matrix in the canonical basis of an essentially self-adjoint operator
in $\ell^{2}(\NN_{0})$. Denote by $\calZ_{n}$ the set
of roots of $P_{n}(x)$ and put
\[
	\Lambda := 
	\bigcap_{N=1}^{\infty} 
	\overline{\bigcup_{n=N}^{\infty} \calZ_{n}}.
\]
Then $\supp \mu \subset \Lambda$ and
\begin{equation}
	\label{eq:lim-w}
	\lim_{n \to \infty}
	\frac{Q_{n}(z)}{P_{n}(z)}
	=
	-w(z)
\end{equation}
locally uniformly on $\CC \setminus \Lambda$.

As a consequence of Markov's Theorem it has been observed in \cite[Eq. (28)]{Stovicek2020}
that
\begin{equation}
	\label{eq:w0-Po-sum}
	w_{n}(z) = 
	-P_{n}(z) 
	\sum_{j=n}^{\infty}
	\frac{1}{a_{j} P_{j}(z) P_{j+1}(z)}, \quad n \geq 0,
\end{equation}
holds on $\CC \setminus \Lambda$, and the series converges locally uniformly on this set. Furthermore, combining \eqref{eq:lim-F}
and \eqref{eq:lim-w} with the defining relation \eqref{eq:def-W}
we immediately get
\[
	\frakW(z)
	=
	-\lim_{n \to \infty}
	\frac{Q_{n}(z)}{P_{n}(0)}
\]
locally uniformly on $\CC \setminus \Lambda$.

By the assumptions formulated in the beginning of this section, $J \geq \gamma$
for some $\gamma > 0$ (of course, the best choice is $\gamma = \zeta_{1}$).
         
By $J^{[k]}$, $k\geq1$, let us denote the Jacobi matrix operator obtained
from $J$ by removing the first $k$ rows and columns. Then $J^{[k]}$
obeys the same assumptions as $J$ does, see \cite[Prop. 4]{Stovicek2022}.
All objects associated with $J^{[k]}$ will be distinguished by the
upper index $[k]$. Particularly this concerns the orthonormal polynomial
sequence $\big( P^{[k]}_n(x) : n \geq 0 \big)$ and the functions of the second kind $\big( w_{n}^{[k]}(z) : n \geq 0 \big)$,
see \eqref{eq:fce-second-kind}, and also the characteristic function
$\mathfrak{F}^{[k]}(z)$.
The upper index $[0]$ corresponds to objects which are related to
the original Jacobi operator $J$.

Some attention should be paid to the polynomials of the second kind,
however. As in \cite{Stovicek2020}, we define them so that
\begin{equation}
	\label{eq:poly-second-kind}
	Q_{0}^{[k]}(x) = Q_{1}^{[k]}(x) = \ldots = Q_{k}^{[k]}(x)=0,\ 
	Q_{k+n}^{[k]}(x) = \frac{1}{a_{k}}\, P_{n-1}^{[k+1]}(x)\,\ \text{for}\ n \geq 1.
\end{equation}

Let us recall that there exists a generalization of Markov's theorem
(\ref{eq:lim-w}) due to Van Assche, see \cite[Thm. 2]{VanAssche1991a}, which claims that,
for $k \geq 0$,
\begin{equation}
	\label{eq:Markov-gen}
	w_{k}(z) = 
	-\lim_{n \to \infty}
	\frac{Q_{n}^{[k]}(z)}{P_{n}(z)} \quad
	\text{locally\ uniformly\ in}\ 
	z \in \CC \setminus [\gamma, +\infty).
\end{equation}
It is also useful to recall that 
\begin{equation}
	\label{eq:wk-wPk-Qk}
	\forall k \geq 0,\ \,
	w_{k}(z) = w(z) P_{k}(z) + Q_{k}(z),
\end{equation}
(see \cite[Eq. (13)]{Stovicek2020}).
This relation should be understood as an equality between two meromorphic
functions on $\CC$. Moreover, the functions of the second
kind form a square summable sequence for all $z$ from the resolvent
set, i.e.
\begin{equation}
	\label{eq:w-vect-ell2}
	\forall z\in\varrho(J),\ \,
	\big(w_{0}(z), w_{1}(z), w_{2}(z), \ldots \big)^{T} \in \ell^{2}(\NN_{0}).
\end{equation}
Furthermore, by \cite[Eq. (26)]{Stovicek2020},
\begin{equation}
	\label{eq:Q-m-n}
	\forall m \geq n \geq 0,\ \,
	Q_{m}^{[n]}(x) = Q_{m}(x) P_{n}(x) - P_{m}(x) Q_{n}(x).
\end{equation}

\begin{theorem} 
It holds true that
\begin{equation}
	\label{eq:w1k-wk-over-w}
	\forall k \geq 0,\,\ 
	w_{k}^{[1]}(z) = -\frac{w_{k+1}(z)}{a_{0} w(z)},
\end{equation}
and
\begin{equation}
	\label{eq:wk-Fchar}
	\forall k \geq 0,\,\ 
	w_{k}(z) = \frac{w_{k}(0) \frakF^{[k+1]}(z)}{\frakF(z)}.
\end{equation}
These equations are both understood as equalities between meromorphic
functions on $\CC$. 
\end{theorem}
\begin{proof}
Regarding the Jacobi matrix operator $J^{[1]}$ with the associated sequence of orthogonal polynomials $(P^{[1]}_n(x) : n\geq0)$, denote temporarily the corresponding orthogonal polynomials of the second kind $\tilde{Q}_n(x)$, $n\geq0$. Note that, in accordance with the notation introduced in \eqref{eq:poly-second-kind}, we have $\tilde{Q}_n(x)=Q^{[1]}_{n+1}(x)$.
Letting $k=1$ in \eqref{eq:poly-second-kind} we get
\[
   \tilde{Q}_0(x) = Q^{[1]}_{1}(x) = 0
   \quad\text{and}\quad
   \tilde{Q}_n(x) = Q^{[1]}_{n+1}(x) = \frac{1}{a_1}\,P^{[2]}_{n-1}(x)
   \ \ \text{for}\ n\geq 1
\]
(observe that this also implies correct degrees: $\deg \tilde{Q}_n(x)=n-1$ for $n\geq1$).
Thus we deduce that when replacing $J$ by $J^{[1]}$ the polynomials of the second kind $(Q_n(x) : n\geq0)$ should be replaced by the sequence $(Q^{[1]}_{n+1}(x) : n\geq0)$ and, more generally, the sequence $(Q^{[k]}_n(x) : n\geq0)$ in \eqref{eq:poly-second-kind} should be replaced by $(Q^{[k+1]}_{n+1}(x) : n\geq0)$.

Applying \eqref{eq:Markov-gen} to the Jacobi matrix operator $J^{[1]}$ we obtain
\[
	w_{k}^{[1]}(z)
	=
	-\lim_{m \to \infty} 
	\frac{Q_{m}^{[k+1]}(z)}{P_{m-1}^{[1]}(z)}
	=
	-\lim_{m \to \infty} 
	\frac{Q_{m}^{[k+1]}(z)}{a_{0} Q_{m}(z)}
\]
locally uniformly on $\CC \setminus [\gamma, +\infty)$. Using
\eqref{eq:Q-m-n}, \eqref{eq:lim-w} and \eqref{eq:wk-wPk-Qk} we
get
\begin{align*}
	w_{k}^{[1]}(z)
        &= - \lim_{m \to \infty}
        \frac{Q_m(z)P_{k+1}(z)-P_m(z)Q_{k+1}(z)}{a_0 Q_m(z)} \\
	&= 
	-\frac{1}{a_{0}}
	\bigg( P_{k+1}(z)
	- Q_{k+1}(z) \lim_{m \to \infty} \frac{P_{m}(z)}{Q_{m}(z)}
	\bigg) \\
	&=
	-\frac{1}{a_{0}}
	\bigg(
	P_{k+1}(z) + \frac{Q_{k+1}(z)}{w(z)}
	\bigg) \\
 	&= 
 	-\frac{w_{k+1}(z)}{a_{0} w(z)}.
\end{align*}
Now it suffices to observe that, under our assumptions, all functions
of the second kind are non-zero meromorphic functions on $\CC$.

From \eqref{eq:poly-second-kind} we also infer that
\[
   \frac{Q_{k+n}^{[k]}(z)}{Q_{k+n}^{[k]}(0)}
   = \frac{P_{n-1}^{[k+1]}(z)}{P_{n-1}^{[k+1]}(0)}
\]
for all $k\in\mathbb{N}_0$, $n\in\mathbb{N}$. Consequently,
\[
   Q_{n}^{[k]}(z) = \frac{P_{n-k-1}^{[k+1]}(z)}{P_{n-k-1}^{[k+1]}(0)}\,
   Q_{n}^{[k]}(0)\ \ \ \text{for}\ n \geq k+1.
\]
Therefore, taking into account also \eqref{eq:lim-F},
\begin{align*}
	w_{k}(z)
	&=
	-\lim_{n \to \infty}
	\frac{Q_{n}^{[k]}(z)}{P_{n}(z)} \\
	&=
	-\lim_{n \to \infty}
	\frac{P_{n-k-1}^{[k+1]}(z)}{P_{n-k-1}^{[k+1]}(0)} 
	\cdot
	\frac{Q_{n}^{[k]}(0)}{P_{n}(0)}
	\cdot
	\frac{P_{n}(0)}{P_{n}(z)} \\
	&=
	\frakF^{[k+1]}(z) w_{k}(0) \frac{1}{\frakF(z)}.
\end{align*}
Again, the limit is locally uniform on $\CC \setminus [\gamma, +\infty)$.
Finally notice that the involved characteristic functions are all
entire functions. 
\end{proof}

\begin{remark*} 
One can also prove \eqref{eq:wk-Fchar} by mathematical induction on $k$, with the aid of \eqref{eq:w1k-wk-over-w}. 
\end{remark*}

\begin{corollary} 
It holds true that
\begin{equation} 
	\label{eq:Fchar-PF1-QF}
	w_{k}(0) \mathfrak{F}^{[k+1]}(z)
	=
	w(0) P_{k}(z) \frakF^{[1]}(z)
	+
	Q_{k}(z) \frakF(z) \quad \text{for } k \geq 0
\end{equation}
and all $z \in \CC$. 
\end{corollary}
\begin{proof} 
In equation~\eqref{eq:wk-wPk-Qk} we substitute for $w_{k}(z)$ and $w(z)\equiv w_{0}(z)$ from~\eqref{eq:wk-Fchar} and get~\eqref{eq:Fchar-PF1-QF}. 
\end{proof}

\begin{corollary} 
For every eigenvalue $\zeta$ of $J$, a corresponding eigenvector can be chosen as the column vector
\begin{equation}
	\label{eq:vect-w0F}
	\big( 
	w(0) \frakF^{[1]}(\zeta),
	w_{1}(0) \frakF^{[2]}(\zeta),
	w_{2}(0) \frakF^{[3]}(\zeta),
	\ldots
	\big)^{T}.
\end{equation}
\end{corollary}
\begin{proof} 
For every $z\in\mathbb{C}$, the column vector
\[
	\bm{P}(z)
	=
	\big(
	P_{0}(z), P_{1}(z), P_{2}(z), \dots
	\big)^{T}
\]
is a formal eigenvector of $J$, i.e. a solution of the equation $(\calJ-z \Id) \bm{P}(z) = \bm{0}$
in $\ell(\NN_0)$, and $x \in \RR$ is an eigenvalue of $J$ if and only if $\bm{P}(x)$ has a finite $\ell^{2}$ norm.

Letting $z = \zeta \in \sigma(J)$ in~\eqref{eq:Fchar-PF1-QF} we see
that the column vector~\eqref{eq:vect-w0F} is a multiple of $\bm{P}(\zeta)$.

Moreover, the first entry in \eqref{eq:vect-w0F}, i.e. $w(0) \frakF^{[1]}(\zeta)$,
is necessarily non-zero. Actually, it is known that the eigenvalues
of $J$ and $J^{[1]}$ interlace \cite[Corollary 10.23]{Lukic2022}. Hence
the zeroes of $\frakF(z)$ and $\frakF^{[1]}(z)$ interlace.
But we have $\frakF(\zeta)=0$. This is also seen from \eqref{eq:wk-Fchar}
for $k=0$ since Weyl's function has a singularity at $z = \zeta$ which
is not removable and the zero of $\frakF(z)$ at $z = \zeta$
is simple. 
\end{proof}

\begin{remark*} 
Note that the column vector \eqref{eq:vect-w0F} has a finite $\ell^{2}$ norm for all $\zeta \in \CC$ and not
only for the eigenvalues of $J$. If $\zeta$ is not an eigenvalue
of $J$ then it follows from \eqref{eq:w-vect-ell2} and \eqref{eq:wk-Fchar}.
\end{remark*}

\begin{theorem} 
We have that
\begin{equation}
	\label{eq:W-series}
	\frakW(z) = w(0) + z \sum_{n=1}^{\infty} w_{n}(0) Q_{n}(z),
\end{equation}
and the series converges locally uniformly on $\CC$. 
\end{theorem}                                     
\begin{proof} 
By Theorem~2 in \cite{Stovicek2020}, for every
Jacobi operator~$J$ fulfilling our assumptions it holds true
that the series $\sum_{n=0}^{\infty} w_{n}(0) P_{n}(z)$ converges locally
uniformly on $\CC$ and the respective characteristic function
$\frakF(z)$ satisfies equation \eqref{eq:Fchar-sum-wP}. In
particular, all these claims are true for $J^{[1]}$ as well.

According to \eqref{eq:w1k-wk-over-w} and \eqref{eq:poly-second-kind},
\[
	w_{n}(0) Q_{n}(z)
	=
	-w(0) w_{n-1}^{[1]}(0) P_{n-1}^{[1]}(z) \quad \text{for } n \geq 1.
\]
Hence the series $\sum_{n=1}^{\infty} w_{n}(0) Q_{n}(z)$ converges
locally uniformly on $\CC$ and
\begin{align*}
	w(0) + z \sum_{n=1}^{\infty} w_{n}(0) Q_{n}(z) 
	&=
	w(0)
	\Bigg(
	1 - 
	z \sum_{n=0}^{\infty} w_{n}^{[1]}(0) P_{n}^{[1]}(z)
	\Bigg) \\
	&= 
	w(0)\mathfrak{F}^{[1]}(z) \\
 	&= 
 	w(z) \frakF(z) \\
 	&= 
 	\frakW(z).
\end{align*}
In the last two equations we have used \eqref{eq:wk-Fchar}, with
$k = 0$, and the very definition of $\frakW(z)$ (see \eqref{eq:def-W}). 
\end{proof}

\begin{remark} 
Finally a short remark concerning Green's function.
It is valid for arbitrary Jacobi matrix $\calJ$. Recall that
Green's function of $\calJ$, called $\calG$, is the
unique semi-infinite strictly lower triangular matrix which is a right
inverse of $\calJ$, i.e. $\calJ \calG = \Id$, see \cite{Geronimo1988}.
Then, for all $z \in \CC$,
\[
	\bm{P}(z) = (I - z \calG)^{-1} \bm{P}(0)
\]
(the inverse $(\Id-z\mathcal{G})^{-1}$ can be defined by the geometric
series) \cite[Eq. (30)]{Stovicek2020}. Similarly,
\[
	\bm{Q}(z) = \big(Q_{0}(z), Q_{1}(z), Q_{2}(z), \ldots \big)^{T}
\]
is a formal solution of the equation $(\calJ - z \Id) \bm{Q}(z) = \bm{e}_{0}$
in $\CC^{\infty}$. The solution is unique with the additional
assumption that the first entry is zero. It holds again true that
\[
	\bm{Q}(z) = (\Id - z \calG)^{-1} \bm{Q}(0).
\]
\end{remark}

\section{Jacobi matrices with trace resolvent related to generators of Birth-Death processes} \label{sec:4}

\subsection{A criterion for the trace class in the case of generators of Birth-Death processes}

In this section we specialize the general results derived in the preceding section to Jacobi matrices corresponding to generators of Birth-Death processes, namely, with Jacobi parameters as given in \eqref{eq:34}.
This way we extend substantially the class of Jacobi operators studied earlier in \cite{Stovicek2022}
which is briefly recalled in Example~\ref{rem:jat22}.

Let us define 
\begin{equation} 
	\label{eq:35}
	\pi_n =
	\begin{cases}
		\dfrac{\lambda_0 \lambda_1 \ldots \lambda_{n-1}}{\mu_1 \mu_2 \ldots \mu_n} & n \geq 1, \\
		\hfil 1 & n=0.
	\end{cases}
\end{equation}

Occasionally, apart of orthonormal polynomials we shall also need a monic version of $(P_n : n \in \NN_0)$. They satisfy
\begin{equation} 
	\label{eq:3}
	p_n(z) = a_0 a_1 \ldots a_{n-1} P_n(z)
\end{equation}
and
\begin{equation} 
	\label{eq:4}
	z p_n(z) = a_{n-1}^2 p_{n-1}(z) + b_n p_n(z) + p_{n+1}(z), \quad n \geq 1
\end{equation}
with the initial conditions
\begin{equation} 
	\label{eq:5}
	p_0(z) \equiv 1, \quad p_1(z) = z-b_0.
\end{equation}

The following result can be derived from \cite[formula (2.4)]{Karlin1957}. Our proof seems to be new.
\begin{lemma} \label{lem:1}
Consider Jacobi parameters \eqref{eq:34} with $\mu_0 \geq 0$. Then for any $n \geq 0$
\begin{align} 
	\label{eq:36}
	p_n(0) 
	&= 
	(-1)^n \lambda_0 \ldots \lambda_{n-1} 
	\bigg( 1 + \mu_0 \sum_{k=0}^{n-1} \frac{1}{\lambda_k \pi_k} \bigg) \\
	\label{eq:37}
	&= 
	(-1)^n \lambda_0 \ldots \lambda_{n-1} 
	\bigg( 1 + \sum_{k=0}^{n-1} \frac{\mu_0 \ldots \mu_k}{\lambda_0 \ldots \lambda_k} \bigg),
\end{align}
and consequently,
\begin{equation} \label{eq:38}
	P_n(0) = 
	(-1)^n \sqrt{\pi_n}
	\bigg( 1 + \mu_0 \sum_{k=0}^{n-1} \frac{1}{\lambda_k \pi_k} \bigg).
\end{equation}
\end{lemma}
\begin{proof}
By \eqref{eq:4} and \eqref{eq:5} we have
\begin{equation} \label{eq:39}
	z p_n(z) = p_{n+1}(z) + (\mu_n + \lambda_n) p_n(z) + \lambda_{n-1} \mu_n p_{n-1}(z), \quad n \geq 1
\end{equation}
with the initial conditions: $p_0(z) \equiv 1, p_1(z) = z - \lambda_0 - \mu_0$. We shall only prove that \eqref{eq:37} holds, because then \eqref{eq:36} and \eqref{eq:38} are simple consequences of \eqref{eq:34}, \eqref{eq:35} and \eqref{eq:3}.

Notice that the equation \eqref{eq:39} can be rewritten in the form
\begin{equation} \label{eq:40}
	\begin{pmatrix}
		p_{n}(z) \\
		p_{n+1}(z)
	\end{pmatrix}
	=
	B_n(z)
	\begin{pmatrix}
		p_{n-1}(z) \\
		p_n(z)
	\end{pmatrix}, \quad n \geq 1,
\end{equation}
where
\[
	B_n(z) =
	\begin{pmatrix}
		0 & 1 \\
		-\lambda_{n-1} \mu_n & z - \lambda_n-\mu_n
	\end{pmatrix}.
\]
Notice that for any $n \geq 1$,
\begin{align} 
	\label{eq:41}
	B_n(0) 
	\begin{pmatrix}
		1 \\
		-\lambda_{n-1}
	\end{pmatrix}
	&=
	\begin{pmatrix}
		0 & 1 \\
		-\lambda_{n-1} \mu_n & -\lambda_n - \mu_n
	\end{pmatrix}
	\begin{pmatrix}
		1 \\
		-\lambda_{n-1}
	\end{pmatrix} 
	=
	-\lambda_{n-1}
	\begin{pmatrix}
		1 \\
		-\lambda_n
	\end{pmatrix}, \\
	\label{eq:42}
	B_n(0) 
	\begin{pmatrix}
		0 \\
		1
	\end{pmatrix}
	&=
	\begin{pmatrix}
		1 \\
		-\lambda_n - \mu_n
	\end{pmatrix}
	=
	\begin{pmatrix}
		1 \\
		-\lambda_n
	\end{pmatrix}
	-\mu_n
	\begin{pmatrix}
		0 \\
		1
	\end{pmatrix}.
\end{align}
We claim that for any $n \geq 0$,
\begin{equation}
	\label{eq:43}
	\begin{pmatrix}
		p_n(0) \\
		p_{n+1}(0)
	\end{pmatrix}
	=
	(-1)^n \lambda_0 \ldots \lambda_{n-1}
	\bigg( 
	1 + 
	\sum_{k=0}^{n-1} \frac{\mu_0 \ldots \mu_k}{\lambda_0 \ldots \lambda_k} 
	\bigg)
	\begin{pmatrix}
		1 \\
		-\lambda_{n}
	\end{pmatrix}
	+
	(-1)^{n+1} \mu_0 \ldots \mu_n
	\begin{pmatrix}
		0 \\
		1
	\end{pmatrix}.
\end{equation}
This claim holds true for $n=0$. Assume that \eqref{eq:43} holds for $n=m$, we shall prove that then \eqref{eq:43} holds for $n=m+1$.
We compute using \eqref{eq:40} and \eqref{eq:43}, and then \eqref{eq:41} and \eqref{eq:42}
\begin{align*}
	\begin{pmatrix}
		p_{m+1}(0) \\
		p_{m+2}(0)
	\end{pmatrix}
	&=
	B_{m+1}(0)
	\begin{pmatrix}
		p_m(0) \\
		p_{m+1}(0)
	\end{pmatrix} \\
	&=
	(-1)^m \lambda_0 \ldots \lambda_{m-1}
	\bigg( 
	1 + 
	\sum_{k=0}^{m-1} \frac{\mu_0 \ldots \mu_k}{\lambda_0 \ldots \lambda_k} 
	\bigg)
	B_{m+1}(0)
	\begin{pmatrix}
		1 \\
		-\lambda_{m}
	\end{pmatrix}
	+ 
	(-1)^{m+1} \mu_0 \ldots \mu_m
	B_{m+1}(0)
	\begin{pmatrix}
		0 \\
		1
	\end{pmatrix} \\
	&=
	(-1)^{m+1} \lambda_0 \ldots \lambda_{m-1} \lambda_m
	\bigg( 
	1 + 
	\sum_{k=0}^{m} \frac{\mu_0 \ldots \mu_k}{\lambda_0 \ldots \lambda_k} 
	\bigg)
	\begin{pmatrix}
		1 \\
		-\lambda_{m+1}
	\end{pmatrix}
	+ (-1)^{m+2} \mu_0 \ldots \mu_{m+1} 
	\begin{pmatrix}
		0 \\
		1
	\end{pmatrix},
\end{align*}
which shows that then \eqref{eq:43} holds true also for $n=m+1$. Thus, by induction the proof of \eqref{eq:43} is complete.
\end{proof}

\begin{lemma}\label{thm:L1} 
Assume, in addition to the assumptions
of Lemma~\ref{lem:1}, that 
\begin{equation} \label{eq:44}
	\sum_{n=0}^\infty 
	\pi_n(1 + \fraks_n)^2 = \infty, 
	\quad \text{where} \quad
	\fraks_n := 
	\mu_0 
	\sum_{k=0}^{n-1} 
	\frac{1}{\lambda_k \pi_k}.
\end{equation}
Then $J$ is self-adjoint and positive-semidefinite.
Set
\[ \fraks_\infty := \lim_{n \to \infty} \fraks_n.\]
Then, for $n \in \NN_{0}$, we have
\begin{equation}
	\label{eq:s-infty-finite}
	-\Bigg(
	\sum_{j=n}^{\infty}
	\frac{1}{a_{j} P_{j}(0) P_{j+1}(0)}
	\Bigg)
	P_{n}(0)
	=
	(-1)^{n}
	\frac{\sqrt{\pi_{n}}}{1 + \fraks_{\infty}} 
	\sum_{j=n}^{\infty}
	\frac{1}{\lambda_{j} \pi_{j}}
\end{equation}
if $\fraks_{\infty} < \infty$, and
\begin{equation}
	\label{eq:s-infty-inf}
	-\Bigg(
	\sum_{j=n}^{\infty} 
	\frac{1}{a_{j}P_{j}(0)P_{j+1}(0)}
	\Bigg)
	P_{n}(0)
	=
	(-1)^{n}
	\frac{\sqrt{\pi_{n}}}{\mu_{0}}
\end{equation}
 if $\fraks_{\infty} = \infty$.
\end{lemma}    
\begin{proof} 
One can employ formula~\eqref{eq:38}. Likewise in the proof of Theorem \ref{thm:T1},
\eqref{eq:44} implies that $J$ is self-adjoint.
Moreover, a simple consequence of \cite[Thm. 1]{Karlin1957} is that $J$ is positive-semidefinite.

If $\mu_{0} = 0$ then $P_n(0)=(-1)^n\sqrt{\pi_n}$ and
$\fraks_{n} = 0$ for any $n \geq 0$. Thus, $\fraks_{\infty} = 0$
and \eqref{eq:s-infty-finite} is immediately seen to hold. Note that
\[
	a_{j} = \lambda_{j} \sqrt{\frac{\pi_{j}}{\pi_{j+1}}}\,,
\]
For $\mu_{0} > 0$ we have
\[
	-\frac{1}{a_{j} P_{j}(0) P_{j+1}(0)}
	=
	\frac{1}{\lambda_{j} \pi_{j} (1 + \fraks_{j})(1 + \fraks_{j+1})}
	=
	\frac{1}{\mu_{0}}
	\bigg(
	\frac{1}{1 + \fraks_{j}} - \frac{1}{1 + \fraks_{j+1}}
	\bigg).
\]
Whence
\begin{align*}
	-\Bigg(
	\sum_{j=n}^{\infty}
	\frac{1}{a_{j} P_{j}(0) P_{j+1}(0)}
	\Bigg)
	P_{n}(0) 
	&=
	(-1)^{n}
	\frac{\sqrt{\pi_{n}} (1 + \fraks_{n})}{\mu_{0}}
	\lim_{N \to \infty}
	\bigg(
	\frac{1}{1 + \fraks_{n}} - \frac{1}{1+\fraks_{N}}
	\bigg) \\
	&= 
	(-1)^{n}
	\frac{\sqrt{\pi_{n}}}{\mu_{0}}
	\bigg(
	1 - \frac{1 + \fraks_{n}}{1 + \fraks_{\infty}}
	\bigg).
\end{align*}
Depending on the value of $\fraks_{\infty}$ we get either \eqref{eq:s-infty-finite}
or \eqref{eq:s-infty-inf}. 
\end{proof}
\begin{remark*}
If $\mu_0 > 0$, then \eqref{eq:44} is in fact \emph{necessary} for self-adjointness of $J$, see \cite[formula (3.83)]{Iglesia2021}.
\end{remark*}

A straightforward combination of Theorem~\ref{thm:T1}, Lemma~\ref{lem:1}
and Lemma~\ref{thm:L1} yields the following theorem.

\begin{theorem} \label{thm:2}
Let $J$ be the Jacobi operator corresponding to \eqref{eq:34}. If
\[
	\sum_{n=0}^{\infty} \pi_{n}(1 + \fraks_{n})^{2} = \infty,
\]
then $J$ is a positive-semidefinite self-adjoint operator and $0 \notin \sigmaP(J)$.
Moreover, if $\fraks_{\infty} < \infty$ then
\begin{align} 
	\nonumber
    \Tr(J^{-1})
    &=
    \frac{1}{1 + \fraks_{\infty}}
    \sum_{n=0}^{\infty} 
    \pi_{n}(1 + \fraks_{n})
    \sum_{j=n}^{\infty}
    \frac{1}{\lambda_{j} \pi_{j}} \\
    \label{eq:45a}
    &=
    \frac{1}{1 + \fraks_{\infty}}
    \sum_{j=0}^{\infty}
    \frac{1}{\lambda_{j} \pi_{j}}
    \sum_{n=0}^{j} 
    \pi_{n} (1 + \fraks_{n}),
\end{align}
and if $\fraks_{\infty} = \infty$ then
\begin{equation} \label{eq:45b}
    \Tr(J^{-1})
    =
    \frac{1}{\mu_{0}}
    \sum_{n=0}^{\infty} \pi_{n}(1 + \fraks_{n}).
\end{equation}
\end{theorem}

\begin{example}\label{rem:jat22}
In \cite{Stovicek2022} the following Jacobi parameters
\[
	a_n = \kappa \tilde{a}_n, \quad b_n = \tilde{a}_n + \kappa^2 \tilde{a}_{n-1}, \quad n \in \NN_0,
\]
where $\tilde{a}_{-1} := 0, \kappa \in (0,1)$ and $(\tilde{a}_n : n \in \NN_0)$ is a positive sequence such that
\begin{equation} \label{eq:46}
	\sum_{n=0}^\infty \frac{1}{\tilde{a}_n} < \infty.
\end{equation}
Notice that by setting
\[
	\lambda_n = \tilde{a}_n, \quad \mu_{n} = \kappa^2 \tilde{a}_{n-1}
\]
leads to the setup \eqref{eq:34} with $\mu_0=0$. Thus, by \eqref{eq:35} we have for $n \geq 1$,
\[
	\pi_n 
	= 
	\frac{\lambda_0 \lambda_1 \ldots \lambda_{n-1}}{\mu_1 \mu_2 \ldots \mu_n} 
	=
	\frac{1}{\kappa^{2n}}.
\]
Therefore, if $0 \notin \sigmaEss(J)$, then the identity \eqref{eq:45a} reads
\[
	\Tr (J^{-1})
	= \sum_{j=0}^\infty
	\frac{1}{\tilde{a}_j} \sum_{n=0}^j \kappa^{2(j-n)}
	= \sum_{j=0}^\infty
	\frac{1}{\tilde{a}_j} \frac{1 - \kappa^{2j+2}}{1-\kappa^{2}}.
\]
which, thanks to \eqref{eq:46}, is finite.
\end{example}

\subsection{The characteristic function} \label{sec:5:2}
In this subsection we are still focusing on Jacobi operators with the parameters as given in \eqref{eq:34}.
To simplify the discussion we shall assume that
\[
	\sum_{j=0}^{\infty}
	\frac{1}{\lambda_{j} \pi_{j}} < \infty.
\]
In addition to formula~\eqref{eq:38}, we also have
\begin{equation}
	\label{eq:mu0-wn0}
	w_{n}(0) = 
	(-1)^{n} \sqrt{\pi_{n}}
	\Bigg(
	1 + 
	\mu_{0} \sum_{j=0}^{\infty}
	\frac{1}{\lambda_{j} \pi_{j}}
	\Bigg)^{-1}
	\sum_{j=n}^{\infty}
	\frac{1}{\lambda_{j} \pi_{j}}.
\end{equation}
This equation is obtained as a combination of \eqref{eq:w0-Po-sum}
with \eqref{eq:s-infty-finite}.

With regard to applications of Lemma~\ref{thm:Fchar-aux} to follow
it is convenient to set
\begin{equation}
    \label{eq:SN}
	S_{n}
	=
	\sum_{j=n}^{\infty}
	\frac{1}{\lambda_{j}\pi_{j}}, \quad n \geq 0.
\end{equation}
Hence $\fraks_{n} = \mu_{0}\,(S_{0}-S_{n})$, see~\eqref{eq:44}. Then, by \eqref{eq:38} and \eqref{eq:mu0-wn0}, we can write
\begin{equation}
    \label{eq:Pn0-wn0}
	P_{n}(0)
	=
	(-1)^{n} \sqrt{\pi_{n}}
	\big(
	1+\mu_{0} (S_{0}-S_{n})
	\big), \quad
	w_{n}(0)
	=
	(-1)^{n}
	\frac{\sqrt{\pi_{n}} S_{n}}{1 + \mu_{0} S_{0}},
\end{equation}
and
\[
	\kappa_{n} 
	= 
	w_{n}(0)P_{n}(0)
	=
	\frac{1 + \mu_{0}(S_{0}-S_{n})}{1 + \mu_{0} S_{0}} 
	\pi_{n} S_{n}.
\]
Further we put, by definition,
\[
	S_{-1}
	=
	\frac{1}{\mu_{0}}
	+
	S_{0}
\]
if $\mu_{0} > 0$, and $S_{-1} = \infty$ if $\mu_{0} = 0$. Then
\[
	\frac{1+\mu_{0} (S_{0}-S_{n})}{1+\mu_{0} S_{0}}
	=
	1 - \frac{\mu_{0} S_{n}}{1 + \mu_{0} S_{0}}
	=
	1 - \frac{S_{n}}{S_{-1}}.
\]
Hence
\begin{equation}
 \label{eq:Pn0-kappan0}
    P_n(0) = (-1)^n \sqrt{\pi_n}\; (1+\mu_0 S_0)
    \bigg(1 - \frac{S_{n}}{S_{-1}}\bigg),\quad
	\kappa_{n}
	=
	\bigg(
	1 - \frac{S_{n}}{S_{-1}}
	\bigg)
	\pi_{n} S_{n}.
\end{equation}

Further we intend to apply Lemma \ref{thm:Fchar-aux}. By (\ref{eq:Pn0-wn0}) and
(\ref{eq:Pn0-kappan0}) we have
\begin{equation}
    \label{eq:bm-kj}
    \bm{k}_j = \sqrt{\kappa_j}
    = \sqrt{\bigg(1 - \frac{S_{j}}{S_{-1}}\bigg)\pi_j S_j}.
\end{equation}
and
\begin{equation}
    \label{eq:Kmn}
	K_{m,n}
	=
	-\sqrt{\frac{1 - {S_{n}}/{S_{-1}}}{1 - {S_{m}}/{S_{-1}}}}
	\sqrt{\pi_{m} S_{m} \pi_{n} S_{n}}
	\bigg(
	1-\frac{S_{m}}{S_{n}}
	\bigg) \quad \text{for } m>n,
\end{equation}
see \eqref{eq:bm-k}, \eqref{eq:K}.
Suppose $\ell\in\mathbb{N}_0$. Starting from
\[
    \bm{k}^T\cdot K^\ell\cdot \bm{k}
    = \sum_{j_1=0}^\infty \sum_{j_2=0}^\infty \cdots \sum_{j_{\ell+1}=0}^\infty
    \bm{k}_{j_{\ell+1}} K_{j_{\ell+1},j_{\ell}} \,\cdots\, K_{j_2,j_1} \bm{k}_{j_1}
\]
we can use (\ref{eq:bm-kj}), (\ref{eq:Kmn}), take into account that
$K_{m,n}=0$ for $m\leq n$ and get
\begin{align*}
    \bm{k}^T\cdot K^\ell\cdot \bm{k}
    =\ & (-1)^\ell \sum_{0\leq j_1<j_2<\dots j_{\ell+1}}
    \pi_{j_{\ell+1}}S_{j_{\ell+1}} \ldots \pi_{j_2}S_{j_2} \pi_{j_1}S_{j_1}
    \bigg(1 - \frac{S_{j_{\ell+1}}}{S_{j_{\ell}}}\bigg)
    \cdots \bigg(1 - \frac{S_{j_2}}{S_{j_1}}\bigg) \\
    &\times \sqrt{1 - \frac{S_{j_{\ell+1}}}{S_{-1}}}
    \,\sqrt{\frac{1 - {S_{j_\ell}}/{S_{-1}}}{1 - {S_{j_{\ell+1}}}/{S_{-1}}}}
    \,\cdots\,
    \sqrt{\frac{1 - {S_{j_1}}/{S_{-1}}}{1 - {S_{j_2}}/{S_{-1}}}}
    \,\sqrt{1 - \frac{S_{j_1}}{S_{-1}}}
\end{align*}
Now, expanding $(\Id-zK)^{-1}$ into a geometric series in $z$ in \eqref{eq:Fchar}
(and writing $m$ instead of $\ell+1$) we obtain a formula for the characteristic function
\begin{equation}
	\label{eq:Fchar-B-D}
	\frakF(z) 
	=
	1 + \sum_{m=1}^{\infty}
	(-1)^{m} z^{m}
	\sum_{0 \leq j_{1} < j_{2} < \ldots < j_{m}}
	\pi_{j_1}S_{j_1}
	\pi_{j_2}S_{j_2}
	\dots
	\pi_{j_m} S_{j_m}
	\bigg(
	1-\frac{S_{j_{1}}}{S_{-1}}
	\bigg)
	\bigg(
	1-\frac{S_{j_{2}}}{S_{j_{1}}}
	\bigg)
	\cdots
	\bigg(
	1 - \frac{S_{j_{m}}}{S_{j_{m-1}}}
	\bigg).
\end{equation}

Particularly let us consider the case when 
\begin{equation}
    \label{eq:mu0}
    \mu_{0}=0
\end{equation}
(and so $S_{-1}=\infty$). Then (see~\eqref{eq:38},
\eqref{eq:mu0-wn0} with $\mu_{0} = 0$)
\begin{equation}
    \label{eq:63}
    P_{n}(0)
    =
    (-1)^{n} \sqrt{\pi_{n}}, \quad
    w_{n}(0)
    =
    (-1)^{n}\sqrt{\pi_{n}} S_{n}, \quad \text{and} \quad
    \kappa_{n}
    =
    \pi_{n}S_{n}.
\end{equation}
Using Lemma~\ref{thm:Fchar-aux} we can express the characteristic
function as
\[
	\frakF(z) 
	= 
	1 + \sum_{m=1}^{\infty}
	(-1)^{m} z^{m}
	\sum_{0 \leq j_{1} < j_{2} < \ldots < j_{m}}
	\pi_{j_{1}} \pi_{j_{2}}
	\dots
	\pi_{j_{m}}
	(S_{j_{1}} - S_{j_{2}})
	\cdots
	(S_{j_{m-1}} - S_{j_{m}})
	S_{j_{m}}.
\]
This can be rewritten as follows,
\begin{equation}
	\label{eq:charF-general-mu0-0}
	\frakF(z) 
	= 
	1 + \sum_{m=1}^{\infty}
	(-1)^{m} z^{m}
	\sum_{0 \leq j_{1} \leq k_{1} < j_{2} \leq k_{2} < \ldots \leq  k_{m-1} < j_{m} \leq k_{m}}
	\frac
	{\pi_{j_{1}}\pi_{j_{2}}
	\dots
	\pi_{j_{m}}}
	{\lambda_{k_{1}} \pi_{k_{1}}
	\dots
	\lambda_{k_{m-1}} \pi_{k_{m-1}} \lambda_{k_{m}} \pi_{k_{m}}}.
\end{equation}
Denote
\begin{equation}
\label{eq:Tk}
	T_{k}
	=
	\sum_{j=0}^{k} \pi_{j}, \quad k \geq 0.
\end{equation}
Changing the order of summation we get
\begin{equation}
	\label{eq:Fchar-T}
	\frakF(z)
	=
	1 + \sum_{m=1}^{\infty}
	(-1)^{m} z^{m}
	\sum_{0 \leq k_{1} < k_{2} < \ldots < k_{m}}
	\frac
	{T_{k_{1}} 
	(T_{k_{2}}-T_{k_{1}}) 
	\dots 
	(T_{k_{m}} - T_{k_{m-1}})}
	{\lambda_{k_{1}} \pi_{k_{1}}
	\dots
	\lambda_{k_{m-1}} \pi_{k_{m-1}}
	\lambda_{k_{m}} \pi_{k_{m}}}.
\end{equation}

\subsection{The function $\mathfrak{W}(z)$} \label{sec:5:3}
Let us make a few comments on computing the function $\frakW(z)$, as given
in \eqref{eq:def-W} and in \eqref{eq:W-series}, but still assuming \eqref{eq:mu0}.
$J^{[1]}$ again stands for the first associated matrix to $J$ obtained by removing the first row and first column, and $\lambda_{n}^{[1]}$,
$\mu_{n}^{[1]}$ for the corresponding Birth-Death parameters. In
a similar fashion we introduce the symbol $\pi_{n}^{[1]}$ etc. Then
\[
	\lambda_{n}^{[1]}
	=
	\lambda_{n+1}, \quad
	\mu_{n}^{[1]}
	=
	\mu_{n+1}, \quad 
	\pi_{n}^{[1]}
	=
	\pi_{n+1} \frac{\mu_{1}}{\lambda_{0}}.
\]
Furthermore,
\[
	S_{n}^{[1]}
	=
	\sum_{j=n}^{\infty}
	\frac{1}{\lambda_{j}^{[1]} \pi_{j}^{[1]}}
	=
	\frac{\lambda_{0}}{\mu_{1}} S_{n+1},
\]
and also ($\pi_{0}=1$)
\[
	S_{-1}^{[1]}
	=
	\frac{1}{\mu_{0}^{[1]}} + S_{0}^{[1]}
	=
	\frac{1}{\mu_{1}} + \frac{\lambda_{0}}{\mu_{1}} S_{1} 
	=
	\frac{\lambda_{0}}{\mu_{1}}
	\bigg(
	\frac{1}{\lambda_{0}\pi_{0}} + S_{1}
	\bigg) 
	=
	\frac{\lambda_{0}}{\mu_{1}}\,S_{0}.
\]

Regarding the characteristic function of $J^{[1]}$, since $\mu_{0}^{[1]} \neq 0$,
we can use formula \eqref{eq:Fchar-B-D}, namely
\begin{align*}
	\frakF^{[1]}(z) 
	=
	1 
	&+ 
	\sum_{m=1}^{\infty}
	(-1)^{m}z^{m} \\
	&\times
	\sum_{0 \leq j_{1} < j_{2} < \ldots < j_{m}}
	\pi_{j_{1}}^{[1]} S_{j_{1}}^{[1]}
	\dots
	\pi_{j_{m}}^{[1]} S_{j_{m}}^{[1]}
	\Bigg(
	1 - \frac{S_{j_{1}}^{[1]}}{S_{-1}^{[1]}}
	\Bigg)
	\Bigg(
	1 - \frac{S_{j_{2}}^{[1]}}{S_{j_{1}}^{[1]}}
	\Bigg)
	\cdots
	\Bigg(
	1 - \frac{S_{j_{m}}^{[1]}}{S_{j_{m-1}}^{[1]}}
	\Bigg).
\end{align*}
Hence
\[
	\frakF^{[1]}(z) 
	= 
	1 + \frac{1}{S_{0}}
	\sum_{m=1}^{\infty}
	(-1)^{m} z^{m}
	\sum_{1 \leq j_{1} < j_{2} < \ldots < j_{m}}
	\pi_{j_{1}}
	\dots
	\pi_{j_{m}}
	(S_{0} - S_{j_{1}})
	(S_{j_{1}} - S_{j_{2}})
	\cdots
	(S_{j_{m-1}} - S_{j_{m}})
	S_{j_{m}}.
\]
Substituting for $S_{j}$ we get
(note that $S_{0}=w(0)\equiv w_0(0)$, see \eqref{eq:63} with $\pi_0=1$, as defined in \eqref{eq:35})
\begin{equation}
\label{eq:Wz}    
	\frakW(z) 
	= 
	S_{0} \frakF^{[1]}(z)
	= 
	\sum_{m=0}^{\infty} 
	(-1)^{m}z^{m}
	\sum_{0 \leq k_{1} < j_{1} \leq k_{2} < j_{2} \ldots \leq k_{m} < j_{m} \leq k_{m+1}}
	\frac
	{\pi_{j_{1}} \pi_{j_{2}}
	\dots
	\pi_{j_{m}}}
	{\lambda_{k_{1}} \pi_{k_{1}}
	\dots
	\lambda_{k_{m}} \pi_{k_{m}}
	\lambda_{k_{m+1}} \pi_{k_{m+1}}}
\end{equation}
(one may wish to compare this equation to~\eqref{eq:charF-general-mu0-0}).
To avoid possible misunderstanding concerning the notation used in (\ref{eq:Wz}) let us note that the first term in the sum in this equation, for $m=0$, equals
\[
\sum_{k_1=0}^\infty \frac{1}{\lambda_{k_1}\pi_{k_1}} = S_0,
\]
as it should be. Other terms may be already clear.
The equation can be also rewritten with the aid of the symbol $T_{k}$,
as defined in \eqref{eq:Tk},
\[
	\frakW(z)
	=
	\sum_{m=0}^{\infty}
	(-1)^{m}z^{m}
	\sum_{0 \leq k_{1} < \ldots < k_{m} < k_{m+1}}
	\frac
	{(T_{k_{2}} - T_{k_{1}})
	\cdots
	(T_{k_{m+1}} - T_{k_{m}})}
	{\lambda_{k_{1}} \pi_{k_{1}}
	\dots
	\lambda_{k_{m}} \pi_{k_{m}} 
	\lambda_{k_{m+1}}\pi_{k_{m+1}}}
\]
(again, this may be compared to~\eqref{eq:Fchar-T}).

\section{Some results for positive Jacobi matrices with compact inverses}
\label{sec:4new}

\subsection{On the convergence of the zeros of orthogonal polynomials} \label{sec:4.1}
Suppose that Jacobi operator $J$ is positive. Let $J_F$ be its \emph{Friedrichs extension}(\footnote{If $J$ is self-adjoint, then its every self-adjoint extension is equal to $J$. In particular, then $J=J_F$.}) (see, e.g. \cite[Chapter 8.1]{Schmudgen2017}). Suppose that $J_F^{-1}$ is compact and let $\sigma(J_F) = \{\zeta_j : j \geq 1\}$ be ordered by
\begin{equation} 
	\label{eq:26}
	0 < \zeta_1 < \zeta_2 < \ldots.
\end{equation}
Let $(p_n : n \geq 0)$ be the sequence of monic orthogonal polynomials associated with $\calJ$.
Then by \cite[Theorem I.5.2]{Chihara1978} we have that all the zeros of $(p_n : n \geq 1)$ are 
real, simple and located in $(\zeta_1, \infty)$. In particular, 
\begin{equation} 
	\label{eq:10}
	\forall n \geq 1 \ 
	\forall z \in \CC \setminus (\zeta_1, \infty), \quad
	p_n(z) \neq 0.
\end{equation}
Thus we have
\begin{equation} 
	\label{eq:11}
	p_n(z) = (z - x_{1;n}) \ldots (z - x_{n;n}),
\end{equation}
where
\begin{equation}
	\label{eq:27}
	0 < \zeta_1 < x_{1;n} < x_{2;n} < \ldots < x_{n;n}.
\end{equation}

The following result will be instrumental in what follows.
\begin{theorem} \label{thm:3}
Suppose that Jacobi operator $J$ is positive and let $J_F$ be its Friedrichs extension. Let $J_F^{-1}$ be compact.
Let $\sigma(J_F) = \{\zeta_j : j \geq 1\}$ be ordered as in \eqref{eq:26} and let the monic orthogonal polynomials
be written as in \eqref{eq:11} and \eqref{eq:27}. Then
\begin{align}
	\label{eq:14}
	&\zeta_j < x_{j;n}, \quad j=1,2,\ldots,n;\ n \geq 1,\\
	\label{eq:15}
	&\zeta_j = \lim_{n \to \infty} x_{j;n}, \quad j=1,2,\ldots.
\end{align}
\end{theorem}
\begin{proof}
First of all, by \cite[Theorem II.4.1]{Chihara1978} we get that(\footnote{By $\sharp X$ we denote the number of elements in $X$.})
\begin{equation} \label{eq:28}
	\sharp \big( (x_{i;n}, x_{i+1;n}) \cap \sigma(J_F) \big) \geq 1, \quad i=1,2,\ldots,n-1.
\end{equation}
By \eqref{eq:27} and \eqref{eq:26} it easily implies \eqref{eq:14}.
Now if $J$ is not-self-adjoint, i.e. $J \neq J_F$, the formula~\eqref{eq:15} has been proven in \cite[p. 479]{Chihara1968}.
On the other hand, if $J$ is self-adjoint, then formula~\eqref{eq:28} together with Theorems I.5.3 and II.4.3 from \cite{Chihara1978} imply~\eqref{eq:15}.
\end{proof}

\subsection{The zero counting measures} \label{sec:zero}
Suppose that we are working in the setup of Section~\ref{sec:4.1}. 
In this section we shall analyse weak convergence of the sequence of \emph{zero counting measures} of $P_n$ defined by
\begin{equation} 
    \label{eq:II:5}
    \nu_n = \sum_{j=1}^n \delta_{x_{j;n}},
\end{equation}
see formulas \eqref{eq:11} and \eqref{eq:27}.

\begin{theorem} \label{thm:5}
Suppose that a Jacobi matrix $J$ is positive and $J_F^{-1}$ is compact. 
Let $\sigma(J_F) = \{\zeta_j : j \geq 1\}$ be ordered as in \eqref{eq:26}.
Suppose that for some continuous non-decreasing function $w : \RR \to (0, \infty)$ we have
\begin{equation} 
    \label{eq:56}
    \sum_{j=1}^\infty \frac{1}{w(\zeta_j)} < \infty.
\end{equation}
Then for any(\footnote{By $\calC_b(\RR)$ we denote the space of continuous bounded functions on the real line equipped with the supremum norm $\norm{\cdot}_\infty$.}) $f \in \calC_b(\RR)$ such that $\norm{w f}_\infty < \infty$ we have
\begin{equation} 
    \label{eq:55a}
    \lim_{n \to \infty} 
    \int_\RR f \ud \nu_n =
    \int_\RR f \ud \nu_\infty,
\end{equation}
where $\nu_n$ is defined in \eqref{eq:II:5} and
\begin{equation} 
    \label{eq:55b}
    \nu_\infty = \sum_{j=1}^\infty \delta_{\zeta_j}.
\end{equation}
\end{theorem}
\begin{proof}
In order to prove \eqref{eq:55a} it is enough to show that
\begin{equation}
    \label{eq:59}
    \lim_{n \to \infty} 
    \sum_{j=1}^n 
    \frac{\tilde{f}(x_{j;n})}{w(x_{j;n})} 
    =
    \sum_{j=1}^\infty 
    \frac{\tilde{f}(\zeta_j)}{w(\zeta_j)}, \quad \tilde{f} \in \calC_b(\RR).
\end{equation}
Indeed, to obtain~\eqref{eq:55a} it is enough to apply \eqref{eq:59} for $\tilde{f}(x) = w(x) f(x)$. 
Notice that by \eqref{eq:14} and the fact that the function $w$ is non-decreasing we get
\begin{equation} 
    \label{eq:57}
    \frac{\tilde{f}(x_{j;n})}{w(x_{j;n})} 
    \leq 
    \frac{\norm{\tilde{f}}_\infty}{w(x_{j;n})} 
    \leq 
    \frac{\norm{\tilde{f}}_\infty}{w(\zeta_j)}.
\end{equation}
Moreover, by \eqref{eq:15} and the continuity of $\tilde{f}$ and $w$ we get that
\begin{equation} 
    \label{eq:58}
    \lim_{n \to \infty} 
    \frac{\tilde{f}(x_{j;n})}{w(x_{j;n})} =
    \frac{\tilde{f}(\zeta_j)}{w(\zeta_j)}.
\end{equation}
By combining \eqref{eq:58}, \eqref{eq:57} and \eqref{eq:56}, by Dominated Convergence Theorem, we get~\eqref{eq:57} what concludes the proof.
\end{proof}

\begin{remark} \label{rem:2}
The hypothesis~\eqref{eq:56} is only important for continuous functions $f$ with non-compact support. Indeed, if(\footnote{The set $\calC_c(\RR) \subset \calC_b(\RR)$ consists of compactly supported continuous functions on $\RR$.}) $f \in \calC_c(\RR)$, then the convergence~\eqref{eq:55a} holds by \eqref{eq:15} as each side of \eqref{eq:55a} involves only sums with uniformly bounded number of terms.
\end{remark}

\begin{example}
Suppose that $J$ is positive and $J_F^{-1}$ belongs to the Schatten class $\calS_k$ for some $k \geq 1$. Let $\sigma(J_F) = \{\zeta_j : j \geq 1\}$ be ordered as in \eqref{eq:26}. Then since
\[
    \sum_{j=1}^\infty \frac{1}{\zeta_j^k} < \infty
\]
we can take in Theorem~\ref{thm:5}
\[
    w(x) = 
    \begin{cases}
        (1 + x^2)^{k/2} & x > 0, \\
        1 & x \leq 0.
    \end{cases}
\]
\end{example}

\subsection{The Christoffel--Darboux kernels} \label{sec:CDKernel}
Let us recall that given $J$, let $\tilde{J}$ be a self-adjoint extension on $\ell^2(\NN_0)$ of $J$. 
Then we can define a Borel probability measure $\tilde{\mu}$ by the formula
\[
	\tilde{\mu}(\cdot) = \langle \bm{e}_0, E_{\tilde{J}}(\cdot) \bm{e}_0 \rangle,
\]
where $E_{\tilde{J}}$ is the projection-valued spectral measure of $\tilde{J}$.
By Riesz's theorem (see e.g. \cite[Theorem 2.14, p. 62]{Shohat1943}) the sequence 
$(P_n : n \geq 0)$ forms an orthonormal basis in $L^2(\tilde{\mu})$. Now, for any $n \geq 1$ let
\[
	K_n(x,y) = \sum_{j=0}^{n-1} p_j(x) p_j(y), \quad x,y \in \RR
\]
be the $n$th \emph{Christoffel--Darboux kernel}. In this section we shall study the weak convergence of the sequence
\begin{equation}
	\label{eq:50}
	\ud \eta_n = K_n(x,x) \ud \tilde{\mu}(x).
\end{equation}

Let us start with a general result.
\begin{theorem} \label{thm:4}
Let $\tilde{J}$ be a self-adjoint extension on $\ell^2(\NN_0)$ of $J$. If $\sigmaEss(\tilde{J}) = \emptyset$, then
\begin{equation} \label{eq:IV:2}
	\lim_{n \to \infty} 
	\int_\RR f \ud \eta_n 
	= 
	\int_\RR f \ud \eta_\infty
\end{equation}
for any function $f \in L^1(\eta_\infty)$, where
\begin{equation} 
	\label{eq:60}
	\eta_\infty = \sum_{\zeta \in \sigma(\tilde{J})} \delta_\zeta.
\end{equation}
\end{theorem}
\begin{proof}
Recall that $\sigma(\tilde{J}) = \supp(\tilde{\mu})$. Since $\sigmaEss(\tilde{J}) = \emptyset$, it follows that $\tilde{\mu}$ is purely discrete. Let us recall that
\[
	K_\infty(\zeta, \zeta) := 
	\lim_{n \to \infty} K_n(\zeta,\zeta) 
	= 
	\frac{1}{\tilde{\mu}(\{\zeta\})}, \quad \zeta \in \supp(\tilde{\mu}),
\]
see e.g. \cite[Corollary 2.6, p. 45]{Shohat1943} in case $J=\tilde{J}$ and \cite[Theorem 2.13, p. 60]{Shohat1943} otherwise.
Thus,
\[
	\int_\RR f \ud \eta_n = 
	\sum_{\zeta \in \sigma(\tilde{J})} 
	f(\zeta) 
	\frac{K_n(\zeta,\zeta)}{K_\infty(\zeta, \zeta)}.
\]
Since 
\[
	|f(\zeta)|
	\frac{K_n(\zeta,\zeta)}{K_\infty(\zeta, \zeta)}
	\leq 
	|f(\zeta)|
\]
and
\[
	\lim_{n \to \infty} 
	\frac{K_n(\zeta,\zeta)}{K_\infty(\zeta, \zeta)} = 1
\]
we get \eqref{eq:IV:2} by Dominated Convergence Theorem.
\end{proof}

Let $\nu_n$ be defined in \eqref{eq:II:5}.
It is well-known that the sequences of probability measures $(\tfrac{1}{n} \nu_n : n \geq 1)$ and $(\tfrac{1}{n} \eta_n : n \geq 1)$ are equiconvergent in the sense
\[
	\lim_{n \to \infty} 
	\bigg| 
	\frac{1}{n} \int_\RR f \ud \nu_n - 
	\frac{1}{n} \int_\RR f \ud \eta_n 
	\bigg| = 0, \quad f \in \calC_c(\RR),
\]
see \cite[Theorem 2.1]{Geronimo1988a}. Notice that since $(\nu_n)$ does not depend on the self-adjoint extension of $J$, then the limit (if exists) of $(\tfrac{1}{n} \eta_n)$ also does not depend on the self-adjoint extension of $J$. The following remark states that in general a similar result \emph{does not hold} for the sequences $(\nu_n : n \geq 1)$ and $(\eta_n : n \geq 1)$ themselves.
\begin{remark}
Suppose that $J$ is positive, let $J_F$ be its Friedrichs extension and $\tilde{J}$ be an arbitrary self-adjoint extension of $J$ on $\ell^2(\NN_0)$. Suppose that $J_F^{-1}$ is compact. 
Then by Remark~\ref{rem:2}
\[
	\lim_{n \to \infty} 
	\int_\RR f \ud \nu_n 
	= 
	\int_\RR f \ud \nu_\infty, \quad f \in \calC_c(\RR),
\]
where $\nu_\infty$ is defined in \eqref{eq:55b}. If $J$ is self-adjoint, then $J=J_F=\tilde{J}$. So by Theorem~\ref{thm:4}
\begin{equation} 
	\label{eq:61}
	\lim_{n \to \infty} 
	\int_\RR f \ud \eta_n 
	= 
	\int_\RR f \ud \eta_\infty, \quad f \in \calC_c(\RR),
\end{equation}
where $\eta_\infty$ is defined in \eqref{eq:60}, and since $\tilde{J} = J_F$ we get $\nu_\infty = \eta_\infty$. On the other hand, if $J$ is not self-adjoint, then $\sigmaEss(\tilde{J})=\emptyset$, see e.g. \cite[Theorem 7.7(i)]{Schmudgen2017}. So we still have the convergence~\eqref{eq:61} for $\eta_n$ obtained from $\tilde{\mu}$ via the formula~\eqref{eq:50}. Since all self-adjoint extensions on $\ell^2(\NN_0)$ of $J$ have disjoint spectra, see e.g. \cite[Theorem 7.7(ii)]{Schmudgen2017}, we get $\nu_\infty = \eta_\infty$ iff $\tilde{J} = J_F$.
\end{remark}

\subsection{The regularized characteristic function for Schatten classes} \label{sec:2}
Let a sequence of distinct real numbers $(\zeta_j : j \geq 1)$ ordered by $0 < |\zeta_1| \leq |\zeta_2| < \ldots$
satisfies for some $k \in \{ 1,2,\ldots \}$
\begin{equation} 
	\label{eq:29}
	\sum_{j=1}^\infty \frac{1}{|\zeta_j|^k} < \infty.
\end{equation}
Then one can define an entire function $\frakF_k : \CC \to \CC$ by the formula
\begin{equation} 
	\label{eq:8}
	\frakF_k(z) =
	\prod_{i=1}^\infty 
	E_{k-1} \bigg(\frac{z}{\zeta_i} \bigg),
\end{equation}
where \emph{$n$th elementary factor} is defined by by
\begin{equation} 
    \label{eq:7}
    E_n(w) = (1-w) \exp \bigg(\sum_{j=1}^n \frac{1}{j} w^j \bigg), \quad w \in \CC,
\end{equation}
(see e.g. \cite[Definition VII.5.10]{Conway1978}). We have $\frakF_k(z) = 0$ iff $z=\zeta_j$ for some $j$ (see, e.g. \cite[Theorem VII.5.12]{Conway1978}).

Suppose now that we working in the setup of Section~\ref{sec:4.1}. Suppose additionally, that $J_F^{-1}$ exists and belongs to some Schatten class $\calS_k$ for some $k \in \{1,2, \ldots \}$. It means that for $\sigma(J_F) = \{\zeta_j : j \geq 1\}$ ordered as in \eqref{eq:26} we have \eqref{eq:29}. In particular, we can define \emph{regularized characteristic function} of $J_F$ by the formula~\eqref{eq:8}. Notice that we also have
\begin{equation} 
    \label{eq:62}
    \frakF_k(z) = \det \big[ E_{k-1} \big( z J_F^{-1} \big) \big], 
\end{equation}
where $\det$ is the \emph{Fredholm determinant}. Recall that since the operator $J_F^{-1}$ is bounded (even compact and self-adjoint) and $E_{k-1}$ is an entire function, the operator $E_{k-1} \big( z J_F^{-1} \big)$ can be defined by e.g. holomorphic functional calculus. Note that the equality between definitions \eqref{eq:8} and \eqref{eq:62} follows from, e.g. \cite[Theorem 9.2(a)]{SimonTrace}.

Let us recall that monic orthogonal polynomials (see \eqref{eq:3}) satisfy
\begin{equation} 
    \label{eq:6}
    p_n(z) = \det(z \Id_n - \calJ_n),
\end{equation}
where $\Id_n$ is $n \times n$ identity matrix and for any $n \geq 1$ the truncation of $\calJ$, $\calJ_n$, is defined by
\[
    \calJ_n = 
     \begin{pmatrix}
     b_0 & a_0 &    &       &\ \\
     a_0 & b_1 & a_1 &        &  \\
        & a_1 & \ddots & \ddots     &  \\
        &    & \ddots & \ddots &  a_{n-2}\\
      &  &   & a_{n-2} & b_{n-1}
    \end{pmatrix}.
\]

Having fixed the notation we can formulate our main result of this section.
\begin{theorem} \label{thm:1}
Suppose that $J$ is positive and $J_F^{-1}$ belongs to the Schatten class $\calS_{k}$ for some $k \in \{ 1, 2, \ldots \}$. 
Then
\begin{equation} 
    \label{eq:9}
    \lim_{n \to \infty}
    \det \Big[ E_{k-1} \big(z \calJ^{-1}_n \big) \Big]
    =
    \frakF_k(z)
\end{equation}
locally uniformly with respect to $z \in \CC$.
\end{theorem}
\begin{proof}
By \eqref{eq:10} for any $n \geq 1$ the matrix $\calJ_n$ is invertible (cf. \eqref{eq:6}) and by \eqref{eq:11}
\begin{equation} 
    \label{eq:12}
    \det 
    \Big[ E_{k-1} \big(z \calJ_n^{-1} \big) \Big]
    = 
    \prod_{j=1}^n E_{k-1} 
    \bigg( \frac{z}{x_{j;n}} \bigg).
\end{equation}
In view of \eqref{eq:8}, in order to prove \eqref{eq:9} we need to prove that
\[
    \lim_{n \to \infty}
    \prod_{j=1}^n E_{k-1} 
    \bigg( \frac{z}{x_{j;n}} \bigg)
    =
    \prod_{j=1}^\infty E_{k-1} 
    \bigg( \frac{z}{\zeta_j} \bigg)
\]
locally uniformly with respect to $z \in \CC$. 

Let us take $R > 0$. 
Notice that by~\eqref{eq:15} for any $j \geq 1$,
\begin{equation} 
    \label{eq:16}
    \lim_{n \to \infty} 
    \frac{z}{x_{j;n}} 
    = 
    \frac{z}{\zeta_j}
\end{equation}
uniformly for $|z| \leq R$. Notice that for any fixed $1 \leq j_0 \leq n$
\[
    \prod_{j=1}^n 
    E_{k-1} \bigg( \frac{z}{x_{j;n}} \bigg)
    =
    \prod_{j=1}^{j_0-1} 
    E_{k-1} \bigg( \frac{z}{x_{j;n}} \bigg)
    \cdot
    \prod_{j=j_0}^n 
    E_{k-1} \bigg( \frac{z}{x_{j;n}} \bigg).
\]
Thus, by continuity of $E_{k-1}$ and \eqref{eq:16} we need to only prove that for some $j_0 \geq 1$
\begin{equation} 
    \label{eq:18}
    \lim_{n \to \infty}
    \prod_{j=j_0}^n 
    E_{k-1} \bigg( \frac{z}{x_{j;n}} \bigg)
    =
    \prod_{j=j_0}^\infty
    E_{k-1} \bigg( \frac{z}{\zeta_j} \bigg)
\end{equation}
uniformly for $|z| \leq R$. Now, let $j_0 \geq 1$ be such that $\tfrac{R}{\lambda_{j}} \leq \tfrac{1}{2}$ for any $j \geq j_0$. By~\eqref{eq:14} we also have $\tfrac{R}{x_{j;n}} \leq \tfrac{1}{2}$ for any $j \geq j_0$ and $n \geq 1$. Since
\[
    |E_{k-1}(w)-1| \leq |w|^k, \quad |w| \leq \frac{1}{2},
\]
see e.g. \cite[Lemma VII.5.11]{Conway1978}, we get by~\eqref{eq:14}
\[
    \bigg| E_{k-1} \bigg( \frac{z}{x_{j;n}} \bigg) -1 \bigg| 
    \leq 
    \bigg( \frac{R}{x_{j;n}} \bigg)^k
    \leq 
    \bigg( \frac{R}{\zeta_j} \bigg)^k.
\]
Since $J^{-1} \in \calS_k$ the last sequence is summable. Thus, in view of uniform Tannery's theorem for infinite products (for a pointwise version see e.g. \cite[Theorem 7.7]{Loya2017}) together with \eqref{eq:16} we get \eqref{eq:18}, what concludes the proof.
\end{proof}

\begin{remark} \label{rem:1}
It is known (see e.g. \cite[Theorem 2.11]{Beckermann2001}) that if $J$ is not self-adjoint, then $J_F^{-1}$ is a Hilbert--Schmidt operator, so in such a case in Theorem~\ref{thm:1} one can take any $k \geq 2$.
\end{remark}

Using \eqref{eq:11} and \eqref{eq:12} one can verify that for any $k \geq 1$,
\begin{equation} \label{eq:54}
	\det \Big[ E_{k-1} \big( z \calJ_n^{-1} \big) \Big] 
	=
	\frac{p_n(z)}{p_n(0)} 
	\exp \bigg( \sum_{j=1}^{k-1} \frac{z^j}{j} \tr \big( \calJ_n^{-j} \big) \bigg).
\end{equation}
In order to compute the value of \eqref{eq:54} in terms of $p_n$ we need the following proposition.
\begin{proposition} \label{prop:1}
Let $n \geq 1$. Define for $1 \leq k \leq n$
\[
	\frakp_k = \tr \big( \calJ_n^{-k} \big), \quad 
	\frake_k = \frac{(-1)^k}{k!} \frac{p_n^{(k)}(0)}{p_n(0)}.
\]
Then we have
\begin{equation} 
	\label{eq:19}
	\frakp_k =
	\sum_{i=1}^{k-1} 
	(-1)^{i-1}
	\frakp_{k-i} 
	\frake_{i} 
	+ (-1)^{k-1} k \frake_k, \quad 1 \leq k \leq n.
\end{equation}
In particular,
\begin{equation} 
	\label{eq:20}
	\frakp_1 = \frake_1, \quad 
	\frakp_2 = \frake_1^2 - 2 \frake_2, \quad
	\frakp_3 = \frake_1^3 - 3 \frake_1 \frake_2 + 3 \frake_3
\end{equation}
\end{proposition}
\begin{proof}
Notice that by \eqref{eq:6} and \eqref{eq:11} we have
\begin{equation} 
	\label{eq:21}
	\frakp_k = \sum_{i=1}^n \frac{1}{x_{i;n}^k}.
\end{equation}
Next, by \eqref{eq:11}
\[
	\frac{p_n(z)}{p_n(0)} = 
	\prod_{i=1}^n
	\Big( 1 - \frac{z}{x_{i;n}} \Big).
\]
By induction on $1 \leq k \leq n$ one can prove that
\[
	\bigg[ 
	\prod_{i=1}^n
	\bigg( 1 - \frac{z}{x_{i;n}} \bigg) 
	\bigg]^{(k)}(z)
	=
	\sum_{\substack{I \subset \{1,\ldots,n \}\\ \sharp I = k}}
	\prod_{i \in I} \frac{-1}{x_{i;n}}
	\cdot
	\prod_{i \in \{1,\ldots,n\} \setminus I} 
	\bigg(1 - \frac{z}{x_{i;n}} \bigg).
\]
Thus,
\begin{equation} 
	\label{eq:22}
	\frake_k = 
	\frac{(-1)^k}{k!}
	\frac{p_N^{(k)}(0)}{p_n(0)} = 
	\sum_{1 \leq i_1 < \ldots < i_k \leq n}
	\prod_{j=1}^k \frac{1}{x_{i_j;n}}.
\end{equation}
Now, in view of \eqref{eq:21} and \eqref{eq:22}, the formula \eqref{eq:19} is just Newton's identity, see e.g. \cite{Mead1992}. Then the formula \eqref{eq:20} is an application of \eqref{eq:20}.
\end{proof}

In view of \eqref{eq:54} and Proposition~\ref{prop:1} we can rewrite the 
argument of the exponential on left-hand side of \eqref{eq:9} in terms of $p_n^{(j)}(0)/p_n(0)$. 
In particular, it leads to
\begin{align}
	\label{eq:23}
	\det \Big[ E_{0} \big( z\calJ_n^{-1} \big) \Big] 
	&= \frac{p_n(z)}{p_n(0)}\\
	\label{eq:24}
	\det \Big[ E_{1} \big( z\calJ_n^{-1} \big) \Big] 
	&= \frac{p_n(z)}{p_n(0)} 
	\exp \bigg( -z \frac{p_n'(0)}{p_n(0)} \bigg) \\
	\label{eq:25}
	\det \Big[ E_{2} \big( z\calJ_n^{-1} \big) \Big] 
	&= \frac{p_n(z)}{p_n(0)}
	\exp 
	\bigg( 
	-z \frac{p_n'(0)}{p_n(0)} + 
	\frac{z^2}{2} \bigg\{\bigg( \frac{p_n'(0)}{p_n(0)} \bigg)^2 - 
	\frac{p_n''(0)}{p_n(0)} \bigg\} 
	\bigg).
\end{align}

Let us recall that in \cite[Section 3]{Stovicek2020} the following particular case was proven. 
\begin{corollary} \label{cor:1}
Suppose that $J$ is positive and $J^{-1}$ is a trace class operator. Then
\begin{equation} \label{eq:49}
	\lim_{n \to \infty} \frac{p_n(z)}{p_n(0)} = \frakF_1(z)
\end{equation}
locally uniformly with respect to $z \in \CC$.
\end{corollary}
\begin{proof}
Apply Theorem~\ref{thm:1} with $k=1$ and use \eqref{eq:23}.
\end{proof}

\begin{remark} 
Theorem~\ref{thm:3} together with \cite[Theorem 2]{Chihara1972} imply that if the operator $J$ is positive with $J_F^{-1}$ compact, then the the left-hand side of \eqref{eq:49} is convergent locally uniformly on $\CC$ to an entire function with the zero set $\sigma(J_F)$ iff $J_F^{-1}$ is trace class. In \cite{Chihara1972} the right-hand side of \eqref{eq:49} was not identified.
\end{remark}

\subsection{Importance of the positivity} \label{sec:pos}
Suppose that $J$ is a self-adjoint Jacobi matrix with $b_n \equiv 0$. 
Let $(P_n : n \geq 0)$ be the corresponding orthonormal polynomials and let $\mu$ be their unique orthonormality measure, i.e. defined in \eqref{eq:measureOrth}. Then it is well-known that the measure $\mu$ is symmetric around $0$ and (see e.g. formulas (1.6)--(1.9) in \cite{Chihara1982}) that
\[
	P_{2n}(z) = \tilde{P}_n(z^2), \quad 
	P_{2n+1}(z) = z \hat{P}_n(z^2), \quad n \geq 0,
\]
where the sequences $(\tilde{P}_n : n \geq 0), (\hat{P}_n : n \geq 0)$ are orthonormal with respect to the measures
\begin{equation} \label{eq:47}
	\tilde{\mu}(X) = 
	\mu
	\big(
		\{ \pm \sqrt{x} : x \in X \cap [0, \infty) \}
	\big) 
	\quad \text{and} \quad
	\hat{\mu}(X) = \int_X x \ud \tilde{\mu}(x),
\end{equation}
respectively. 

Suppose that $J^{-1}$ is a Hilbert--Schmidt operator, then
\[
	\sum_{j=1}^\infty \frac{1}{\zeta_j^2} < \infty
\]
and its $2$-regularized characteristic function is equal to
\[
	\frakF_2(z) = 
	\prod_{j=1}^\infty
	\bigg[
	\bigg( 1 - \frac{z}{\zeta_j} \bigg) 
	\exp \Big( \frac{z}{\zeta_j} \Big)
	\bigg]
	\bigg[
	\bigg( 1 - \frac{z}{-\zeta_j} \bigg) 
	\exp \Big( \frac{z}{-\zeta_j} \Big)
	\bigg]
	= 
	\prod_{j=1}^\infty 
	\bigg( 1 - \frac{z^2}{\zeta_j^2} \bigg).
\]
Note that if $\supp \mu = \{ \pm \zeta_j : j \geq 1 \}$, where $0 < \zeta_1 < \zeta_2 < \ldots$ then 
by~\eqref{eq:47} $\supp \tilde{\mu} = \supp \hat{\mu} = \{ \zeta_j^2 : j \geq 1 \}$. 
Consequently, their characteristic functions are equal to
\[
	\tilde{\frakF}_1(z) = \hat{\frakF}_1(z) = 
	\prod_{j=1}^\infty \bigg( 1 - \frac{z}{\zeta_j^2} \bigg).
\]

Suppose that both $\tilde{\mu}$ and $\hat{\mu}$ are determinate (in Example~\ref{ex:1} we show that it is possible). 
Then the corresponding Jacobi matrices $\tilde{J}, \hat{J}$ are positive and their inverses belong to the trace class. 
Thus, by Corollary~\ref{cor:1}, for any fixed $w \in \CC_+$
\[
	\frac{P_{2n}(z)}{P_{2n}(w)} = 
	\frac{\tilde{P}_n(z^2)}{\tilde{P}_n(w^2)} \rightarrow 
	\frac{\tilde{\frakF}_1(z^2)}{\tilde{\frakF}_1(w^2)} =
	\frac{\frakF_2(z)}{\frakF_2(w)}
\]
and 
\[
	\frac{P_{2n+1}(z)}{P_{2n+1}(w)} = 
	\frac{z \hat{P}_n(z^2)}{w \hat{P}_n(w^2)} \rightarrow 
	\frac{z}{w}
	\frac{\hat{\frakF}_1(z^2)}{\hat{\frakF}_1(w^2)} =
	\frac{z}{w} \frac{\frakF_2(z)}{\frakF_2(w)},
\]
Consequently, we have
\begin{equation} 
	\label{eq:48}
	\lim_{n \to \infty}
	\frac{P_{2n}(z)}{P_{2n}(w)} =
	\frac{\frakF_2(z)}{\frakF_2(w)}
	\quad \text{whereas} \quad
	\lim_{n \to \infty}
	\frac{P_{2n+1}(z)}{P_{2n+1}(w)} =
	\frac{z}{w} \frac{\frakF_2(z)}{\frakF_2(w)},
\end{equation}
which are not equal when $z,w \in \CC_+$ are such that $z \neq w$.

\begin{example}[Symmetric Al-Salam--Carlitz II polynomials] \label{ex:1}
Let $a>0$ and $0<q<1$. Let $J$ be the Jacobi matrix with Jacobi parameters
\[
	a_{2n} = \sqrt{a q^{-n}}, \quad 
	a_{2n+1} = \sqrt{q^{-n-1} - 1}, \quad
	b_n \equiv 0, \qquad n \geq 0.
\]
According to \cite[p. 342]{Chihara1982} all of the measures $\mu, \tilde{\mu}, \hat{\mu}$ are determinate iff $aq^2 \geq 1$. Then, $\tilde{J} > 0$ and $\tilde{J}^{-1}$ is trace class. Therefore, in view of \eqref{eq:47}, the operator $J^{-1}$ is Hilbert--Schmidt. Thus, it falls into the setup of this subsection, and consequently, we have~\eqref{eq:48}. 
\end{example}

\section{The modified $q$-Laguerre polynomials} \label{sec:6}

\subsection{Preliminaries} \label{sec:6:1}

In this section we introduce and study the so called modified $q$-Laguerre polynomials.
In contrast to the well established standard $q$-Laguerre polynomials (see~\cite{Moak1981})
they are distinguished by the fact that the respective Hamburger moment problem is determinate.
Furthermore, they correspond to a Birth-Death process and they are derived from
the standard $q$-Laguerre polynomials by a construction based on the notion
of quasi-orthogonal sequence (see~\cite{Bracciali2018}).

Everywhere in this section we assume that $q\in(0,1)$ is a fixed parameter. The notation concerning $q$-symbols and $q$-functions used in this section is in compliance with that introduced in \cite{Gasper2004} which now seems to be widely accepted as standard. The same notation can be also found in \cite{Koekoek2010}. In particular let us recall that the
$q$-Pochhammer symbol, $(a;q)_n$, is defined for any $a \in \CC$ and
$n \in \NN_0 \cup \{\infty\}$ by
\[
    (a;q)_n = \prod_{k=0}^{n-1} (1-aq^k).
\]
Furthermore, the basic hypergeometric function $\,_{r}\phi_{s}$ is defined by the series
\[
   \rphis{r}{s}{a_1,\ldots,a_r}{b_1,\ldots,b_s}{q,z}
   = \sum_{k=0}^\infty
   \frac{(a_1;q)_k\cdots (a_r;q)_k}{(b_1;q)_k\cdots (b_s;q)_k}\,
   (-1)^{(1+s-r)k}q^{(1+s-r)\binom{k}{2}}\,
   \frac{z^k}{(q;q)_k}\,.
\]

The (standard) $q$-Laguerre polynomials are defined by the three-term
recurrence relation (see e.g. \cite[formula (14.21.5)]{Koekoek2010})
\begin{multline*}
	(1-q^{n+1}) L_{n+1}^{(\alpha)}(x;q)
	-
	\big(
	1-q^{n+1}+q (1-q^{n+\alpha})
	\big) L_{n}^{(\alpha)}(x;q) \\
	+
	q (1-q^{n+\alpha}) L_{n-1}^{(\alpha)}(x;q)
	=
	-q^{2n+\alpha+1}xL_{n}^{(\alpha)}(x;q), \quad n \geq 0,
\end{multline*}
with $L_{-1}^{(\alpha)}(x;q)=0$, $L_{0}^{(\alpha)}(x;q) = 1$, and
$\alpha > -1$ being a parameter. In this section we always assume that $-1<q<1$ is a parameter. The normalized recurrence relation
for the respective monic polynomials $p_{n}^{(\alpha)}(x;q)$ then
reads
\begin{multline*}
	p_{n+1}^{(\alpha)}(x;q) 
	+ 
	q^{-2n-\alpha-1} 
	\big(
	1-q^{n+1}
	+
	q (1-q^{n+\alpha})
	\big)
	p_{n}^{(\alpha)}(x;q) \\
	+
	q^{-4n-2\alpha+1} 
	(1-q^{n}) (1-q^{n+\alpha})
	p_{n-1}^{(\alpha)}(x;q)
	= x p_{n}^{(\alpha)}(x;q),
\end{multline*}
where
\begin{equation}
	\label{eq:qLaguerre-monic}
	L_{n}^{(\alpha)}(x;q) 
	= 
	\frac{(-1)^{n} q^{n\,(n+\alpha)}}{(q;q)_{n}}
	p_{n}^{(\alpha)}(x;q).
\end{equation}

Explicitly, see \cite[Sec. 14.21]{Koekoek2010},
\begin{align}
	\nonumber
	L_{n}^{(\alpha)}(x;q) 
	&=
	\frac{(q^{\alpha+1};q)_{n}}{(q;q)_{n}} 
	\rphis{1}{1}{q^{-n}}{q^{\alpha+1}}{q, -q^{n+\alpha+1} x} \\
	\label{eq:qLaguerre}
	&=
	\frac{(q^{\alpha+1};q)_{n}}{(q;q)_{n}}
	\sum_{k=0}^{n}
	\frac{(q^{-n};q)_{k}}{(q^{\alpha+1};q)_{k} (q;q)_{k}}
	q^{k(k-1)/2} (q^{n+\alpha+1} x)^{k},
\end{align}
and
\begin{equation}
	\label{eq:qLaguerre-to-monic}
	p_{n}^{(\alpha)}(x;q)
	=
	(-1)^{n} q^{-n(n + \alpha)}
	(q^{\alpha+1};q)_{n} 
	\rphis{1}{1}{q^{-n}}{q^{\alpha+1}}{q,-q^{n+\alpha+1} x}.
\end{equation}

The following relation derived in~\cite{Moak1981} will be of importance
in the sequel,
\[
	q^{n}L_{n}^{(\alpha)}(x;q) 
	+ 
	L_{n-1}^{(\alpha+1)}(x;q)
	-
	L_{n}^{(\alpha+1)}(x;q) = 0, \quad n \geq 0.
\]
In regard of \eqref{eq:qLaguerre-monic}, for the monic $q$-Laguerre
polynomials this means that
\begin{equation}
	\label{eq:p-monic-id-a}
	p_{n}^{(\alpha)}(x;q) 
	= 
	p_{n}^{(\alpha+1)}(x;q)
	+
	q^{-2n-\alpha} (1-q^{n})
	p_{n-1}^{(\alpha+1)}(x;q), \quad n \geq 0.
\end{equation}

The parameters $a_{n}^{(\alpha)}$, $b_{n}^{(\alpha)}$ in the corresponding
Jacobi matrix are given by
\[
	\big(
	a_{n}^{(\alpha)}
	\big)^{2}
	=
	q^{-4n-2\alpha-3} (1-q^{n+1})(1-q^{n+\alpha+1}), \quad
	b_{n}^{(\alpha)}
	=
	q^{-2n-\alpha-1}
	\big(
	1-q^{n+1} + q (1-q^{n+\alpha})
	\big),
\]
$n \geq 0$. We put, too,
\[
	d_{n}^{(\alpha)} 
	= 
	\big(
	a_{n-1}^{(\alpha)}
	\big)^{2}
	=
	q^{-4n-2\alpha+1}
	(1-q^{n})(1 - q^{n+\alpha}), \quad n \geq 1,
\]
so that the normalized three-term recurrence relation reads
\[
	p_{n+1}^{(\alpha)}(x;q)
	=
	(x - b_{n}^{(\alpha)})
	p_{n}^{(\alpha)}(x;q)
	-
	d_{n}^{(\alpha)} 
	p_{n-1}^{(\alpha)}(x;q).
\]
It turns out that we are dealing with a Jacobi parameters of the form~\eqref{eq:34}, i.e.
\[
	\big( 
	a_{n}^{(\alpha)} 
	\big)^{2} 
	= 
	\lambda_{n}^{(\alpha)} \mu_{n+1}^{(\alpha)}, \quad
	b_{n}^{(\alpha)} 
	= 
	\lambda_{n}^{(\alpha)} + \mu_{n}^{(\alpha)},
\]
where
\[
	\lambda_{n}^{(\alpha)}
	=
	q^{-2n-\alpha-1}(1-q^{n+1}),\quad 
	\mu_{n}^{(\alpha)}=q^{-2n-\alpha}(1-q^{n+\alpha}).
\]

The Hamburger moment problem for the $q$-Laguerre polynomials is
known to be indeterminate. A detailed discussion
of the $N$-extremal orthogonality measures can be found in~\cite{Moak1981}.

The notion of the so-called quasi-orthogonal polynomials of order
$k \in \NN$ turns out to be quite useful for our purposes although
only the order $k = 2$ will be needed in the current paper. Here we
recall a basic result which is a particular case of Theorem 2.1 in~\cite{Bracciali2018}.

\begin{theorem}\label{thm:BMV} 
Let $(p_{n}(x) : n \geq 0)$ be a
monic orthogonal polynomial sequence obeying a three-term recurrence
relation
\[
	p_{n+1}(x)
	=
	(x - b_{n}) p_{n}(x)
	-
	d_{n} p_{n-1}(x) \quad \text{for } n \geq 0,
\]
with $p_{-1}(x) = 0$ (hence, necessarily, $b_{n} \in \RR$ for
$n \geq 0$ and $d_{n} > 0$ for $n \geq 1$). Let $(f_{n} : n \geq 0)$
be a real sequence, with $f_{0} = 0$ and $f_{n} \neq 0$ for $n \geq 1$.
Let us define
\[
	\tilde{p}_{n}(x)
	=
	p_{n}(x)
	+
	f_{n} p_{n-1}(x) \quad \text{for } n \geq 0,
\]
and
\begin{align*}
	\tilde{b}_{n} &= b_{n} + f_{n} - f_{n+1} 
	\quad \text{for } n \geq 0,\\
	\tilde{d}_{n} &= d_{n} + f_{n} (b_{n-1} - b_{n} - f_{n} + f_{n+1})
	\quad \text{for } n \geq 1.
\end{align*}
Then $(\tilde{p}_{n}(x) : n \geq 0)$ is a monic orthogonal polynomial
sequence if and only if
\begin{equation}
	\label{eq:gen-quasi-ortho}
	f_{n+1} 
	= 
	f_{n} + b_{n} - b_{n-1} 
	-\frac{d_{n}}{f_{n}} + \frac{d_{n-1}}{f_{n-1}} \quad \text{for } n \geq 2,
\end{equation}
and $\tilde{d}_{n} > 0$ for $n \geq 1$. If so then $(\tilde{p}_{n}(x) : n \geq 0)$
obeys the three-term recurrence relation
\[
	\tilde{p}_{n+1}(x)
	=
	(x - \tilde{b}_{n}) \tilde{p}_{n}(x)
	-
	\tilde{d}_{n} \tilde{p}_{n-1}(x) \quad \text{for } n \geq 0,
\]
with $\tilde{p}_{-1}(x) = 0$. Moreover, the coefficients $\tilde{d}_{n}$
also satisfy
\[
	\tilde{d}_{n}
	=
	\frac{f_{n}}{f_{n-1}} d_{n-1},\quad n \geq 2.
\]
\end{theorem}

\begin{remark} 
The polynomials $\tilde{p}_{n}(x)$, as defined in
Theorem~\ref{thm:BMV}, are called \emph{quasi-orthogonal polynomials
of order $2$}. This construction can be generalized to any order
$k \in \NN$, however, as explained in \cite{Bracciali2018}.
The case $k = 1$ is trivial as it means that $\tilde{p}_{n}(x) = p_{n}(x)$
for all $n \geq 0$. The particular case of order $2$ has already
been explored earlier in \cite{Maroni1990}. 
\end{remark}

Further we recall a theorem which is a particular case of Theorem~2.2 in \cite{Bracciali2018}, with some details taken from its proof ibidem.

\begin{theorem} \label{thm:BMV2} 
Suppose that the quasi-orthogonal
polynomial sequence $(\tilde{p}_{n}(x) : n \geq 0)$, as defined in Theorem~\ref{thm:BMV}, is in fact a monic orthogonal polynomial sequence.
                                                               
Denote by $\Lambda$ and $\tilde{\Lambda}$ their moment functionals,
and by $(M_{n} : n \geq 0)$ and $(\tilde{M}_{n} : n \geq 0)$ the
moment sequences associated with the polynomial sequences $(p_{n}(x) : n \geq 0)$
and $(\tilde{p}_{n}(x) : n \geq 0)$, respectively. Using the notation of Theorem~\ref{thm:BMV} let                       
\[
	h(x)
	=
	1 - \frac{f_{1} \tilde{b}_{0}}{\tilde{d}_{1}}
	+
	\frac{f_{1}}{\tilde{d}_{1}} x.
\]
Then for every polynomial $f$,
\[
	\Lambda(f) = \tilde{\Lambda}(hf).
\]
Consequently,
\begin{equation}
	\label{eq:Mtld-recur}
	\forall n \geq 0,\ 
	\tilde{M}_{n+1} 
	= 
	\frac{\tilde{d}_{1}}{f_{1}} M_{n}
	+
	\bigg(
	\tilde{b}_{0}-\frac{\tilde{d}_{1}}{f_{1}}
	\bigg) \tilde{M}_{n}.
\end{equation}
\end{theorem}

\subsection{Definition} \label{sec:6:2}

Let
\[
	\gamma(u,v;q)
	=
	\frac{(q^{u};q)_{\infty}}{(q^{u+v};q)_{\infty}}.
\]
In particular, for $v>0$,
\[
	\gamma(0,v;q) = 0.
\]
For $m \in \NN_{0}$ we have $\gamma(u,m;q) = (q^{u};q)_{m}$.
In regard to the $q$-Binomial Theorem telling us that \cite[Eq. (II.3)]{Gasper2004}
\begin{equation}
	\label{eq:q-Binom}
	\rphis{1}{0}{a}{-}{q,z}
	=
	\frac{(az;q)_{\infty}}{(z;q)_{\infty}}, \quad |z|<1,
\end{equation}
we also have, for $u > 0$ and $u+v > 0$,
\begin{equation}
	\label{eq:gamma-phi}
	\gamma(u,v;q)
	=
	\frac{1}{\rphis{1}{0}{q^v}{-}{q, q^u}}
	=
	\rphis{1}{0}{q^{-v}}{-}{q, q^{u+v}}.
\end{equation}
Furthermore, one can rewrite the function $\gamma(u,v;q)$ in terms
of the $q$-Gamma function. Recall that (\cite[Eq. (I.35)]{Gasper2004}
or \cite[Eq. (1.9.1)]{Koekoek2010})
\[
	\Gamma_{q}(z)
	=
	\frac{(q;q)_{\infty}}{(q^{z};q)_{\infty}} (1-q)^{1-z}.
\]
We have
\begin{equation}
	\label{eq:gamma-Gamma}
	\gamma(u,v;q)
	=
	\frac{\Gamma_{q}(u+v)}{\Gamma_{q}(u)} (1-q)^{v}.
\end{equation}

For later convenience, we show some identities for function $\gamma(u,v;q)$.
\begin{lemma} \label{thm:gamma}
For $m\in\mathbb{N}_0$ and $\beta>-1$,
\[
	\sum_{k=0}^{m-1}
	q^{k} \gamma(k+1,\beta;q) 
	= 
	\frac{\gamma(m,\beta+1;q)}{1 - q^{\beta+1}}
\]
and
\begin{equation}
	\label{eq:sum-gamma-infty}
	\sum_{k=0}^{\infty}
	q^{k} \gamma(k+1,\beta;q) 
	=
	\frac{1}{1-q^{\beta+1}}.
\end{equation}
In particular, for $\ell \in \NN_0$,
\begin{equation}
	\label{eq:q-sup-k-sum-finite}
	\sum_{k=0}^{m-1}
	q^{k} (q^{k+1};q)_{\ell} 
	= 
	\frac{(q^{m};q)_{\ell+1}}{1-q^{\ell+1}}            
\end{equation}
and
\begin{equation}
	\label{eq:q-sup-k-sum}
	\sum_{k=0}^{\infty}
	q^{k} (q^{k+1};q)_{\ell}
	=
	\frac{1}{1-q^{\ell+1}}.
\end{equation}
\end{lemma}        
\begin{proof}
Letting $a = q^{\beta+1}$ and $c = q^{-n}$ in the $q$-Vandermonde Sum (see e.g. \cite[Eq. (II.6)]{Gasper2004}) telling us that 
\[
	\rphis{2}{1}{a,q^{-n}}{c}{q,q}
	=
	\frac{(c/a;q)_{n}}{(c;q)_{n}} a^{n},
\]
we get
\begin{equation}
	\label{eq:q-sup-k-q-beta-sum}
	\sum_{k=0}^{n}q^{k}
	\frac{(q^{\beta+1};q)_{k}}{(q;q)_{k}}
	=
	\frac{(q^{\beta+2};q)_{n}}{(q;q)_{n}}, \quad n \geq 0.
\end{equation}
Then
\begin{align*}
	\sum_{k=0}^{m-1}
	q^{k} \gamma(k+1,\beta;q)
	&= 
	\frac{(q;q)_\infty}{(q^{\beta+1};q)_\infty}
	\sum_{k=0}^{m}
	q^{k} \frac{(q^{\beta+1};q)_{k}}{(q;q)_{k}} \\
	&= 
	\frac{(q;q)_\infty}{(q^{\beta+1};q)_\infty}
	\frac{(q^{\beta+2};q)_{m-1}}{(q;q)_{m-1}} \\
	&= 
	\frac{\gamma(m,\beta+1;q)}{1-q^{\beta+1}}.
\end{align*}
In the limit $m \to \infty$ we get \eqref{eq:sum-gamma-infty}.
Equations \eqref{eq:q-sup-k-sum-finite} and \eqref{eq:q-sup-k-sum} are particular cases for $\beta = \ell \in \NN_0$.

Let us note that to show \eqref{eq:q-sup-k-sum} one can alternatively use the $q$-Binomial Theorem \eqref{eq:q-Binom} while letting $a = q^{\ell+1}$ and $z = q$.
\end{proof}

Let
\begin{equation}
	\label{eq:r-n-alpha}
	r_{n}^{(\alpha)}
	=
	\frac{1 - \gamma(n+1,\alpha+1;q)}{1-\gamma(n,\alpha+1;q)}.
\end{equation}
Then
\begin{equation}
	\label{eq:1-r-n-alpha}
	1 - r_{n}^{(\alpha)}
	=
	\frac
	{q^{n}(1 - q^{\alpha+1}) \gamma(n+1,\alpha;q)}
	{1-\gamma(n,\alpha+1;q)}.
\end{equation}
To see it, note that
\[
	\big(1-\gamma(u,v+1;q)\big)
	-
	\big(1-\gamma(u+1,v+1;q)\big)
	=
	q^{u}(1-q^{v+1}) \gamma(u+1,v;q).
\]

It is known that $\lim_{q \uparrow 1} \Gamma_{q}(x) = \Gamma(x)$ \cite[Eq. (I.36)]{Gasper2004}.
Hence from (\ref{eq:gamma-Gamma}) we infer that, for $u \geq 0$, $v > 0$,
\[
	\gamma(u,v;q) 
	\approx
	\frac{\Gamma(u+v)}{\Gamma(u)}
	(1-q)^{v} \quad \text{as } q \uparrow 1.
\]
From here and \eqref{eq:r-n-alpha}, while still assuming that $\alpha>-1$,
one finds that
\begin{equation}
	\label{eq:lim-q-to-1-wn}
	\forall n \in \NN_{0},\ 
	\lim_{q\,\uparrow1}r_{n}^{(\alpha)}=1.
\end{equation}

We define the modified $q$-Laguerre polynomials $\tilde{L}_{n}^{(\alpha)}(x;q)$
as
\[
	\tilde{L}_{n}^{(\alpha)}(x;q)
	=
	q^{n} r_{n}^{(\alpha)} L_{n}^{(\alpha)}(x;q)
	+
	\big( 1-r_{n}^{(\alpha)} \big) L_{n}^{(\alpha+1)}(x;q).
\]
The respective monic polynomials are denoted $\tilde{p}_{n}^{(\alpha)}(x;q)$,
and they fulfill
\begin{equation}
	\label{eq:p-tld-convex}
	\tilde{p}_{n}^{(\alpha)}(x;q)
	=
	r_{n}^{(\alpha)} p_{n}^{(\alpha)}(x;q)
	+
	\big( 1-r_{n}^{(\alpha)} \big) p_{n}^{(\alpha+1)}(x;q).
\end{equation}
Thus the monic modified $q$-Laguerre polynomials are defined as a
convex combination of the ``standard'' monic $q$-Laguerre polynomials.

It is known that \cite{Koekoek2010}
\[
	\lim_{q \uparrow 1}
	L_{n}^{(\alpha)} \big((1-q)x;q\big)
	=
	L_{n}^{(\alpha)}(x)
\]
where $L_{n}^{(\alpha)}(x)$ are the ordinary Laguerre polynomials.
Therefore, in view of \eqref{eq:lim-q-to-1-wn}, we also have
\[
	\lim_{q \uparrow 1}
	\tilde{L}_{n}^{(\alpha)}
	\big((1-q)x;q\big)
	=
	L_{n}^{(\alpha)}(x).
\]

\subsection{Quasi-orthogonal polynomials as Birth-Death polynomials} \label{sec:6:3}

Let
\begin{equation}
	\label{eq:BD-tilde}
	\tilde{\lambda}_{n}^{(\alpha)}
	=
	\frac
	{q^{-2n-\alpha-1}(1-q^{n+\alpha+1})}
	{r_{n}^{(\alpha)}}, \quad
	\tilde{\mu}_{n}^{(\alpha)}
	=
	q^{-2n-\alpha} (1-q^{n}) r_{n}^{(\alpha)},
\end{equation}
where $r_{n}^{(\alpha)}$ is defined in \eqref{eq:r-n-alpha} (hence
$\tilde{\mu}_{0}^{(\alpha)} = 0$).

The modified $q$-Laguerre polynomials in the monic form, $\tilde{p}_{n}^{(\alpha)}(x;q)$,
have been defined in \eqref{eq:p-tld-convex} as a convex combination of
$p_{n}^{(\alpha)}(x;q)$ and $p_{n}^{(\alpha+1)}(x;q)$. But owing to
identity \eqref{eq:p-monic-id-a} we can also write
\begin{equation}
	\label{eq:ptld-quasi-ortho}
	\tilde{p}_{n}^{(\alpha)}(x;q)
	=
	p_{n}^{(\alpha+1)}(x;q)
	+
	\tilde{\mu}_{n}^{(\alpha)} p_{n-1}^{(\alpha+1)}(x;q).
\end{equation}
Hence $(\tilde{p}_{n}^{(\alpha)}(x;q) : n \geq 0)$ is a quasi-orthogonal polynomial
sequence derived from the monic orthogonal polynomial sequence $(p_{n}^{(\alpha+1)}(x;q) : n \geq 0)$.
With this observation, Theorem~\ref{thm:BMV} is directly applicable
to the modified $q$-Laguerre polynomials.

\begin{theorem} 
The quasi-orthogonal polynomials $(\tilde{p}_{n}^{(\alpha)}(x;q) : n \geq 0)$
constitute a monic orthogonal polynomial sequence corresponding to
the Birth-Death process with the parameters \eqref{eq:BD-tilde}.
Hence this polynomial sequence obeys a three-term recurrence relation
\[
	\tilde{p}_{n+1}^{(\alpha)}(x;q)
	=
	\big( x-\tilde{b}_{n}^{(\alpha)} \big) \tilde{p}_{n}^{(\alpha)}(x;q)
	-
	\tilde{d}_{n}^{(\alpha)}\tilde{p}_{n-1}^{(\alpha)}(x;q), 
	\quad n \geq 0,
\]
with $\tilde{p}_{-1}^{(\alpha)}(x;q) = 0$, where the coefficients are
given by $\tilde{b}_{n}^{(\alpha)} = \tilde{\lambda}_{n}^{(\alpha)}+\tilde{\mu}_{n}^{(\alpha)}$,
$n \geq 0$, and
\[
	\tilde{d}_{n}^{(\alpha)}
	=
	\tilde{\lambda}_{n-1}^{(\alpha)}
	\tilde{\mu}_{n}^{(\alpha)}
	=
	q^{-4n-2\alpha+1}(1-q^{n})(1-q^{n+\alpha})
	\frac{r_{n}^{(\alpha)}}{r_{n-1}^{(\alpha)}} \quad n \geq 1.
\]
\end{theorem}
\begin{proof} 
Comparing to the notation used in Theorem~\ref{thm:BMV}
we have
\[
	b_{n} \equiv b_{n}^{(\alpha+1)}, \quad 
	d_{n} \equiv d_{n}^{(\alpha+1)}, \quad
	f_{n} = \tilde{\mu}_{n}^{(\alpha)}.
\]
Thus we have to show that
\[
	\tilde{\mu}_{n+1}^{(\alpha)}
	=
	\tilde{\mu}_{n}^{(\alpha)}
	+
	b_{n}^{(\alpha+1)}
	-
	b_{n-1}^{(\alpha+1)}
	-
	\frac{d_{n}^{(\alpha+1)}}{\tilde{\mu}_{n}^{(\alpha)}}
	+
	\frac{d_{n-1}^{(\alpha+1)}}{\tilde{\mu}_{n-1}^{(\alpha)}}
	\quad \text{for } n \geq 2.
\]
Obviously, it suffices to show that
\begin{equation}
	\label{eq:proof-quasi-ortho-red}
	\tilde{\mu}_{n+1}^{(\alpha)}
	-
	b_{n}^{(\alpha+1)}
	+
	\frac{d_{n}^{(\alpha+1)}}{\tilde{\mu}_{n}^{(\alpha)}}
	=
	0 \quad \text{for } n \geq 1.
\end{equation}
But this equation can be verified by a straightforward computation.

Furthermore, it is also straightforward to verify that
\[
	\frac
	{\tilde{\mu}_{n}^{(\alpha)}}
	{\tilde{\mu}_{n-1}^{(\alpha)}} d_{n}^{(\alpha+1)}
	=
	\tilde{d}_{n}^{(\alpha)} 
	\quad \text{and} \quad
	b_{n}^{(\alpha+1)}
	+
	\tilde{\mu}_{n}^{(\alpha)}
	-
	\tilde{\mu}_{n+1}^{(\alpha)}
	=
	\tilde{b}_{n}^{(\alpha)}.
\]
Finally, one observes that $\tilde{d}_{n}^{(\alpha)} > 0$ for $n \geq 1$.
\end{proof}

Further we can apply Theorem~\ref{thm:BMV2} to determine the moments
associated with the modified $q$-Laguerre polynomials. Let us recall that
the moment sequence $(M_{n}^{(\alpha)} : n \geq 0)$ associated with
the $q$-Laguerre polynomials is well known, see \cite{Moak1981}. When consulting
this source some caution is needed, however, because the notation
in~\cite{Moak1981} slightly differs from the commonly used notation like
that introduced in \cite{Koekoek2010}. Namely, the variable
$x$ in the $q$-Laguerre polynomials is rescaled by the factor $(1-q)$
in the former source, and the same is true for the $q$-Bessel functions.
We have
\begin{equation}
	\label{eq:Mn-qLaguerre}
	\forall n \geq 0,\ 
	M_{n}^{(\alpha)} = q^{-\alpha n-(n+1)n/2}(q^{\alpha+1};q)_{n}.
\end{equation}

\begin{proposition}\label{thm:Mtld-n-a} 
The moments $\tilde{M}_{n}^{(\alpha)}$
associated with the modified $q$-Laguerre polynomials take the values
\[
	\tilde{M}_{0}^{(\alpha)}
	=
	1 \quad \text{and} \quad 
	\tilde{M}_{n}^{(\alpha)}
	=
	q^{-\alpha n-(n+1)n/2}
	\frac{(q^{\alpha+1};q)_{n}}{r_{0}^{(\alpha)}} 
	\quad \text{for } n \geq 1.
\]
\end{proposition}         
\begin{proof} 
Sticking to the notation of Theorem~\ref{thm:BMV}
one can readily check that~\eqref{eq:gen-quasi-ortho} is equivalent
to the equation
\[
	f_{n} - \tilde{b}_{n} + \frac{\tilde{d}_{n+1}}{f_{n+1}}
	=
	f_{n+1} - \tilde{b}_{n+1} + \frac{\tilde{d}_{n+2}}{f_{n+2}},
	\quad n \geq 0.
\]
Furthermore, one finds that
\[
	f_{n} - \tilde{b}_{n} + \frac{\tilde{d}_{n+1}}{f_{n+1}}
	=
	f_{n+1} - b_{n} + \frac{d_{n}}{f_{n}}, \quad n \geq 1.
\]

Dealing with the modified monic $q$-Laguerre polynomials regarded
as quasi-orthogonal polynomials one has to substitute $\tilde{\mu}_{n}^{(\alpha)}$
for $f_{n}$, see~\eqref{eq:ptld-quasi-ortho}. In view of \eqref{eq:proof-quasi-ortho-red}
we get
\[
	\tilde{\mu}_{n}^{(\alpha)}
	-
	\tilde{b}_{n}^{(\alpha)}
	+
	\frac{\tilde{d}_{n+1}^{(\alpha)}}{\tilde{\mu}_{n+1}^{(\alpha)}}
	=
	\tilde{\mu}_{n+1}^{(\alpha)}
	-
	\tilde{b}_{n+1}^{(\alpha)}
	+
	\frac{\tilde{d}_{n+2}^{(\alpha)}}{\tilde{\mu}_{n+2}^{(\alpha)}}
	=
	\tilde{\mu}_{n+2}^{(\alpha)}
	-
	b_{n+1}^{(\alpha+1)}
	+
	\frac{d_{n+1}^{(\alpha+1)}}{\tilde{\mu}_{n+1}^{(\alpha)}}
	=
	0,\quad n \geq 0.
\]
In particular, for $n=0$, $\tilde{\mu}_{n}^{(\alpha)} = 0$ and
\[
	\tilde{b}_{0}^{(\alpha)}
	=
	\frac{\tilde{d}_{1}^{(\alpha)}}{\tilde{\mu}_{1}^{(\alpha)}}
	=
	\frac{q^{-\alpha-1}(1-q^{\alpha+1})}{r_{0}^{(\alpha)}}.
\]
By Theorem~\ref{thm:BMV2}, equation \eqref{eq:Mtld-recur}, we get
\[
	\forall n \geq 0,\ 
	\tilde{M}_{n+1}^{(\alpha)}
	=
	\frac{q^{-\alpha-1}(1-q^{\alpha+1})}{r_{0}^{(\alpha)}}
	M_{n}^{(\alpha+1)}.
\]
In view of \eqref{eq:Mn-qLaguerre}, the proposition readily follows.
\end{proof}

\begin{remark} 
As already noted, the moment problem for the $q$-Laguerre
polynomials is indeterminate. On the other hand, we shall see later
that the moment problem for the modified $q$-Laguerre polynomials
is determinate. The fact that, except the first terms, both moment
sequences are proportional immediately implies that Carleman's sufficient
condition for the determinacy (see e.g. \cite[Theorem 4.3]{Schmudgen2017}) is \emph{not} applicable to the modified $q$-Laguerre
polynomials. 
\end{remark}

\subsection{The Jacobi matrix operator} \label{sec:6:4}

For the Birth-Death polynomials with the parameters \eqref{eq:BD-tilde}
we have
\[
	\tilde{\pi}_{n}^{(\alpha)}
	=
	\prod_{k=0}^{n-1}
	\frac
	{\tilde{\lambda}_{k}^{(\alpha)}}
	{\tilde{\mu}_{k+1}^{(\alpha)}}
	=
	\frac
	{q^{n} (q^{\alpha+1};q)_{n}\big(1-\gamma(1,\alpha+1;q)\big)}
	{(q;q)_{n} 
	\big(1-\gamma(n,\alpha+1;q)\big)
	\big(1-\gamma(n+1,\alpha+1;q)\big)}
\]
and
\begin{equation}
	\label{eq:lbd-pi-Laguerre-mod}
	\tilde{\lambda}_{n}^{(\alpha)} \tilde{\pi}_{n}^{(\alpha)}
	=
	\frac
	{q^{-n-\alpha-1}
	(q^{\alpha+1};q)_{n+1}
	\big(1-\gamma(1,\alpha+1;q)\big)}
	{(q;q)_{n} \big(1-\gamma(n+1,\alpha+1;q) \big)^{2}}.
\end{equation}
The corresponding Jacobi parameters $\tilde{a}_{n}^{(\alpha)}$, $\tilde{b}_{n}^{(\alpha)}$,
$n \geq 0$, are given by
\begin{align*}
	(\tilde{a}_{n}^{(\alpha)})^{2}
	=
	\tilde{\lambda}_{n}^{(\alpha)}
	\tilde{\mu}_{n+1}^{(\alpha)}
	&=
	q^{-4n-2\alpha-3} (1-q^{n+1}) (1-q^{n+\alpha+1})
	\frac{r_{n+1}^{(\alpha)}}{r_{n}^{(\alpha)}}, \\
	\tilde{b}_{n}^{(\alpha)}
	=
	\tilde{\lambda}_{n}^{(\alpha)}
	+
	\tilde{\mu}_{n}^{(\alpha)}
	&=
	q^{-2n-\alpha-1}
	\bigg(
	\frac{(1-q^{n+\alpha+1})}{r_{n}^{(\alpha)}}
	+
	q (1-q^{n}) r_{n}^{(\alpha)}
	\bigg).
\end{align*}
The Jacobi matrix operator with these entries will be denoted $\tilde{J}^{(\alpha)}$.

\begin{lemma} \label{thm:sum-pi} 
For $n \in \NN_{0}$,
\[
	\sum_{j=0}^{n}
	\tilde{\pi}_{j}^{(\alpha)}
	=
	\frac
	{(q^{\alpha+2};q)_{n} \big( 1-\gamma(1,\alpha+1;q) \big)}
	{(q;q)_{n} \big( 1-\gamma(n+1,\alpha+1;q) \big)}.
\]
\end{lemma}
\begin{proof} Note that
\[
	\frac{1}{1-\gamma(j+1,\alpha+1;q)}
	-
	\frac{1}{1-\gamma(j,\alpha+1;q)}
	=
	\frac
	{q^{j}(1-q^{\alpha+1}) \gamma(j+1,\alpha;q)}
	{\big(1-\gamma(j,\alpha+1;q)\big) 
	\big(1-\gamma(j+1,\alpha+1;q)\big)}.
\]
Thus we have
\begin{align*}
	\sum_{j=0}^{n}\tilde{\pi}_{j}^{(\alpha)} 
	&= 
	\big(1-\gamma(1,\alpha+1;q)\big)
	\sum_{j=0}^{n}
	\frac
	{q^{j} (q^{\alpha+1};q)_{j}}
	{(q;q)_{j} 
	\big(1-\gamma(j,\alpha+1;q) \big)
	\big(1-\gamma(j+1,\alpha+1;q) \big)} \\
	&= 
	\big(1-\gamma(1,\alpha+1;q)\big)
	\frac{(q^{\alpha+2};q)_{\infty}}{(q;q)_{\infty}}
	\sum_{j=0}^{n}
	\bigg(
	\frac{1}{1-\gamma(j+1,\alpha+1;q)} 
	-\frac{1}{1-\gamma(j,\alpha+1;q)} 
	\bigg) \\
	&= 
	\big(1-\gamma(1,\alpha+1;q)\big)
	\frac{(q^{\alpha+2};q)_{\infty}}{(q;q)_{\infty}}
	\bigg(
	\frac{1}{1-\gamma(n+1,\alpha+1;q)}-1
	\bigg) \\
	&= 
	\frac{(q^{\alpha+2};q)_{n}}{(q;q)_{n}}
	\frac{1-\gamma(1,\alpha+1;q)}{1-\gamma(n+1,\alpha+1;q)}.
\end{align*}
\end{proof}

\begin{theorem}\label{thm:Jacobi-J-a} 
The Hamburger moment problem
for the modified $q$-Laguerre polynomials is determinate or, equivalently,
the Jacobi operator $\tilde{J}^{(\alpha)}$ is self-adjoint. Moreover,
$\tilde{J}^{(\alpha)}$ is positive, the inversion $(\tilde{J}^{(\alpha)})^{-1}$
exists and belongs to the trace class.
\end{theorem}                          
\begin{proof} We claim that
\[
	\sum_{j=0}^{\infty}
	\tilde{\pi}_{j}^{(\alpha)}
	=
	\infty 
	\quad \text{and} \quad
	\sum_{j=0}^{\infty}
	\frac{1}{\tilde{\lambda}_{j}^{(\alpha)}\tilde{\pi}_{j}^{(\alpha)}}
	\sum_{n=0}^{j}
	\tilde{\pi}_{n}^{(\alpha)} < \infty.
\]
Firstly, by Lemma~\ref{thm:sum-pi},
\[
	\sum_{j=0}^{\infty}
	\tilde{\pi}_{j}^{(\alpha)}
	=
	\lim_{n \to \infty}
	\frac
	{(q^{\alpha+2};q)_{n} \big(1-\gamma(1,\alpha+1;q)\big)}
	{(q;q)_{n} \big(1-\gamma(n+1,\alpha+1;q)\big)}
	=
	\infty.
\]
Indeed, in virtue of \eqref{eq:gamma-phi},
\[
	\lim_{n \to \infty}
	\gamma(n+1,\alpha+1;q)
	=
	\lim_{n \to \infty}
	\frac{1}{\rphis{1}{0}{q^{\alpha+1}}{-}{q,q^{n+1}}} 
	= 
	1.
\]
Further,
\[
	\sum_{j=0}^{\infty}
	\frac{1}{\tilde{\lambda}_{j}^{(\alpha)} \tilde{\pi}_{j}^{(\alpha)}}
	\sum_{n=0}^{j}
	\tilde{\pi}_{n}^{(\alpha)}
	=
	\frac{q^{\alpha+1}}{1-q^{\alpha+1}}
	\sum_{j=0}^{\infty} 
	q^{j} \big(1-\gamma(j+1,\alpha+1;q) \big) < \infty.
\]

The theorem now follows straightforwardly from Theorem~\ref{thm:2}. Note that $\tilde{\mu}_{0}^{(\alpha)}=0$
and therefore $\fraks_{n} \equiv 0$. 
\end{proof}

We can compute the trace of $(\tilde{J}^{(\alpha)})^{-1}$ explicitly. By Theorem~\ref{thm:2} and using
(\ref{eq:sum-gamma-infty}) we obtain
\begin{equation}
	\label{eq:Tr-J-a-inv}
	\Tr \big( (\tilde{J}^{(\alpha)})^{-1} \big)
	=
	\frac{q^{\alpha+1}}{1-q^{\alpha+1}}
	\sum_{j=0}^{\infty}
	q^{j} \big(1-\gamma(j+1,\alpha+1;q)\big)
	=
	\frac{q^{\alpha+2}}{(1-q)(1-q^{\alpha+2})}.
\end{equation}

\subsection{The characteristic function, the orthogonality measure} \label{sec:6:5}

It follows from Theorem \ref{thm:Jacobi-J-a} that the Jacobi operator
$\tilde{J}^{(\alpha)}$ has a trace class resolvent, its spectrum
consists of countably many simple positive eigenvalues, denoted as
$\{\zeta_{k}^{(\alpha)} : k \geq 1\}$, and
\[
	\sum_{k=1}^{\infty}
	\frac{1}{\zeta_{k}^{(\alpha)}}
	=
	\Tr \big( (\tilde{J}^{(\alpha)})^{-1} \big) < \infty.
\]
The eigenvalues are supposed to be ordered increasingly, 
$0 < \zeta_{1}^{(\alpha)} < \zeta_{2}^{(\alpha)} < \zeta_{3}^{(\alpha)} < \ldots$.

The characteristic function of $\tilde{J}^{(\alpha)}$, as introduced
in Section~\ref{sec:2}, reads
\[
	\tilde{\frakF}^{(\alpha)}(z)
	=
	\prod_{k=1}^{\infty}
	\bigg(1-\frac{z}{\zeta_{k}^{(\alpha)}} \bigg).
\]
The function $\tilde{\frakF}^{(\alpha)}$ is entire and, as stated
in Corollary~\ref{cor:1},
\[
	\lim_{n \to \infty} 
	\frac{\tilde{p}_{n}^{(\alpha)}(z;q)}{\tilde{p}_{n}^{(\alpha)}(0;q)}
	=
	\tilde{\mathfrak{F}}^{(\alpha)}(z)
\]
locally uniformly with respect to $z \in \CC$.

The function $\tilde{\frakF}^{(\alpha)}$ turns out to be expressible in
terms of the $q$-Bessel function. Recall that the Jackson $q$-Bessel
functions of the second kind are
\[
	J_{\nu}^{(2)}(x;q)
	:=
	\frac{(q^{\nu+1};q)_{\infty}}{(q;q)_{\infty}}
	\Big( \frac{x}{2} \Big)^{\nu}
	\rphis{0}{1}{-}{q^{\nu+1}}{q,-\frac{q^{\nu+1}x^{2}}{4}}.
\]

\begin{theorem} \label{thm:Fchar-tld-a} 
The characteristic function
of the of the Jacobi operator $\tilde{J}^{(\alpha)}$ equals
\begin{equation}
	\label{eq:Fchar-alpha}
	\tilde{\frakF}^{(\alpha)}(z) =
	\rphis{0}{1}{-}{q^{\alpha+2}}{q,-q^{\alpha+2}z}
	=
	\frac{(q;q)_{\infty}}{(q^{\alpha+2};q)_{\infty}}
	z^{-(\alpha+1)/2} J_{\alpha+1}^{(2)}(2\sqrt{z};q).
\end{equation}
\end{theorem}                                       
\begin{proof} Let us write
\[
	p_{n}^{(\alpha)}(x;q)
	=
	\sum_{k=0}^{n}
	p_{n,k}^{(\alpha)} x^{k}, \quad
	\tilde{p}_{n}^{(\alpha)}(x;q)
	=
	\sum_{k=0}^{n} \tilde{p}_{n,k}^{(\alpha)} x^{k}.
\]
The $q$-Laguerre polynomials are known explicitly, see \eqref{eq:qLaguerre}
and \eqref{eq:qLaguerre-to-monic}. With regard to \eqref{eq:p-tld-convex},
we find that, for $n \geq k \geq 0$,
\begin{align}
	\nonumber
	\frac{\tilde{p}_{n,k}^{(\alpha)}}{\tilde{p}_{n,0}^{(\alpha)}} 
	&= 
	r_{n}^{(\alpha)}
	\frac{p_{n,k}^{(\alpha)}}{\tilde{p}_{n,0}}
	+
	(1-r_{n}^{(\alpha)})
	\frac{p_{n,k}^{(\alpha+1)}}{\tilde{p}_{n,0}} \\
	\label{eq:p-tld-nk-norm}
	&=
	\frac
	{(-1)^{k} q^{k(k-1) + k(\alpha+1)}(q^{n-k+1};q)_{k}}
	{(q^{\alpha+1};q)_{k} (q;q)_{k}}
	\Bigg(
	1-\frac{1-q^{k}}{1-q^{k+\alpha+1}}
	\frac{(q^{n+1};q)_{\infty}}{(q^{n+\alpha+2};q)_{\infty}}
	\Bigg).
\end{align}
Further we write
\[
	\tilde{\frakF}^{(\alpha)}(z)
	=
	\sum_{k=0}^{\infty}
	F_{k} z^{k},
\]
and chose any $\calC^{1}$ positively oriented Jordan curve $C$ with
$0$ in its interior. Then, owing to the locally uniform convergence,
for every $k \in \NN_{0}$ it holds true that
\begin{align*}
	F_{k} &=
	\frac{1}{2\pi i}
	\int_{C}
	\frac{\tilde{\frakF}^{(\alpha)}(z)}{z^{k+1}} \ud z
	=
	\frac{1}{2\pi i}
	\int_{C}
	\lim_{n \to \infty}
	\frac
	{\tilde{p}_{n}^{(\alpha)}(z;q)}
	{\tilde{p}_{n}^{(\alpha)}(0;q) z^{k+1}} \ud z\\
	&= 
	\frac{1}{2\pi i}
	\lim_{n \to \infty}
	\int_{C} 
	\frac
	{\tilde{p}_{n}^{(\alpha)}(z;q)}
	{\tilde{p}_{n}^{(\alpha)}(0;q) z^{k+1}} \ud z
	=
	\lim_{n \to \infty}
	\frac{\tilde{p}_{n,k}^{(\alpha)}}{\tilde{p}_{n,0}^{(\alpha)}}.
\end{align*}
In view of \eqref{eq:p-tld-nk-norm}, we have
\[
	\lim_{n \to \infty}
	\frac
	{\tilde{p}_{n,k}^{(\alpha)}}
	{\tilde{p}_{n,0}^{(\alpha)}}
	=
	\frac
	{(-1)^{k} q^{k(k-1) + k(\alpha+1)}}
	{(q^{\alpha+1};q)_{k}(q;q)_{k}}
	\bigg(
	1-\frac{1 - q^{k}}{1 - q^{k+\alpha+1}}
	\bigg)
	=
	\frac
	{(-1)^{k} q^{k(k+\alpha+1)}}
	{(q^{\alpha+2};q)_{k}(q;q)_{k}}.
\]
Whence
\[
	\tilde{\frakF}^{(\alpha)}(z)
	=
	\sum_{k=0}^{\infty}
	\frac
	{(-1)^{k} q^{k\,(k+\alpha+1)}}
	{(q^{\alpha+2};q)_{k}(q;q)_{k}}
	z^{k}
	=
	\rphis{0}{1}{-}{q^{\alpha+2}}{q,-q^{\alpha+2} z}.
\]
This shows \eqref{eq:Fchar-alpha}. 
\end{proof}
\begin{remark} 
Notice that from \eqref{eq:Fchar-alpha} one can rederive
the trace formula \eqref{eq:Tr-J-a-inv},
\[
	\Tr \big( (\tilde{J}^{(\alpha)})^{-1} \big)
	=
	-\frac{\ud \tilde{\frakF}^{(\alpha)}(0)}{\ud z}
	=
	\frac{q^{\alpha+2}}{(1-q)(1-q^{\alpha+2})}.
\]
\end{remark}

The moment functional $\tilde{\Lambda}^{(\alpha)}$ for the modified
$q$-Laguerre polynomials is described in Proposition~\ref{thm:Mtld-n-a}.
The corresponding orthogonality measure $\tilde{\mu}^{(\alpha)}$
is a unique probability Borel measure on $\RR$ such that for
every polynomial $f$,
\[
	\tilde{\Lambda}^{(\alpha)}(f)
	=
	\int_\RR f(x) \ud \tilde{\mu}^{(\alpha)}(x).
\]
As an immediate corollary of Theorem~\ref{thm:Fchar-tld-a} and some
general results including Eq. \eqref{mu-W-F} we have the following description of $\tilde{\mu}^{(\alpha)}$.

\begin{corollary} 
The orthogonality measure $\tilde{\mu}^{(\alpha)}$
is supported on the set of positive roots of the function $J_{\alpha+1}^{(2)}(2\sqrt{z};q)$.
For every positive root $\zeta$ of $J_{\alpha+1}^{(2)}(2\sqrt{z};q)$,
\[
	\tilde{\mu}^{(\alpha)}(\{\zeta\})
	=
	\Bigg(
	\sum_{n=0}^{\infty}
	\tilde{P}_{n}^{(\alpha)}(\zeta)^{2}
	\Bigg)^{-1}
	=
	-\tilde{\frakW}^{(\alpha)}(\zeta)
	\bigg(
	\frac{\ud \tilde{\frakF}^{(\alpha)}(\zeta)}{\ud z}
	\bigg)^{-1},
\]
where
\[
	\tilde{\frakW}^{(\alpha)}(z)
	=
	\tilde{w}^{(\alpha)}(z)
	\tilde{\frakF}^{(\alpha)}(z)
\]
and $\tilde{w}^{(\alpha)}(z)$ is the corresponding Weyl function.
\end{corollary}

\subsection{The characteristic function continued} \label{sec:6:6}

Applying~\eqref{eq:Fchar-T} to the modified $q$-Laguerre polynomials
we can re-derive formula~\eqref{eq:Fchar-alpha} for the characteristic
function.

In view of Lemma~\ref{thm:sum-pi} and in agreement with definition~\eqref{eq:Tk} we put             
\begin{equation}
	\label{eq:Tk-Laguerre-mod}
	\tilde{T}_{k}^{(\alpha)}
	=
	\frac
	{(q^{\alpha+2};q)_{k} \big(1-\gamma(1,\alpha+1;q)\big)}
	{(q;q)_{k}\big(1-\gamma(k+1,\alpha+1;q)\big)}, \quad k \geq 0.
\end{equation}

\begin{lemma} 
The sequences $\{\tilde{\lambda}_{k}^{(\alpha)} \tilde{\pi}_{k}^{(\alpha)} \}$
and $\{\tilde{T}_{k}^{(\alpha)} \}$, as given in \eqref{eq:lbd-pi-Laguerre-mod}
and \eqref{eq:Tk-Laguerre-mod}, respectively, satisfy the equations
\begin{equation}
	\label{eq:Tk-Tm-sum}
	\sum_{k=0}^{m-1}
	(q^{k};q^{-1})_{\ell}\,
	\frac
	{\tilde{T}_{k}^{(\alpha)}
	(\tilde{T}_{m}^{(\alpha)}-\tilde{T}_{k}^{(\alpha)})}
	{\tilde{\lambda}_{k}^{(\alpha)} \tilde{\pi}_{k}^{(\alpha)}}
	=
	\frac
	{q^{2\ell+\alpha+2} (q^{m};q^{-1})_{\ell+1}}
	{(1-q^{\ell+1}) (1-q^{\ell+\alpha+2})}\,
	\tilde{T}_{m}^{(\alpha)}, \quad m,\ell \geq 0,
\end{equation}
and
\begin{equation}
	\label{eq:Tk-sum}
	\sum_{k=0}^{\infty}
	(q^{k};q^{-1})_{\ell}\,
	\frac
	{\tilde{T}_{k}^{(\alpha)}}
	{\tilde{\lambda}_{k}^{(\alpha)} \tilde{\pi}_{k}^{(\alpha)}}
	=
	\frac
	{q^{2\ell+\alpha+2}}
	{(1-q^{\ell+1}) (1-q^{\ell+\alpha+2})}, \quad \ell \geq 0.
\end{equation}
\end{lemma} 
\begin{proof} 
First let us consider the sum
\[
	\sum_{k=0}^{m-1}
	(q^{k};q^{-1})_{\ell}\,
	\frac
	{\tilde{T}_{k}^{(\alpha)}}
	{\tilde{\lambda}_{k}^{(\alpha)} \tilde{\pi}_{k}^{(\alpha)}}, 
	\quad m,\ell \geq 0.
\]
Notice that this expression vanishes for $m \leq \ell$. Assuming that
$m \geq \ell+1$, the sum equals
\begin{multline*}
	\frac{q^{\alpha+1}}{1 - q^{\alpha+1}}
	\sum_{k=0}^{m-1}
	q^{k} (q^{k};q^{-1})_{\ell}
	\big(1 - \gamma(k+1,\alpha+1;q) \big) \\
	=
	\frac{q^{\ell+\alpha+1}}{1 - q^{\alpha+1}}
	\sum_{k=0}^{m-\ell-1}
	q^{k} (q^{k+\ell};q^{-1})_{\ell}
	\bigg(
	1-
	\frac
	{(q^{k+\ell+1};q)_{\infty}}
	{(q^{k+\ell+\alpha+2};q)_{\infty}}
	\bigg).
\end{multline*}
Notice that $(q^{k+\ell};q^{-1})_{\ell} = (q^{k+1};q)_{\ell}$. Applying
\eqref{eq:q-sup-k-sum-finite} and afterwards \eqref{eq:q-sup-k-q-beta-sum}
we get the expression
\begin{align*}
	\frac{q^{\ell+\alpha+1}}{1-q^{\alpha+1}}
	&\bigg(
	\frac{(q^{m-\ell};q)_{\ell+1}}{1-q^{\ell+1}}
	-\frac{(q;q)_{\infty}}{(q^{\ell+\alpha+2};q)_{\infty}}
	\sum_{k=0}^{m-\ell-1}
	q^{k} \frac{(q^{\ell+\alpha+2};q)_{k}}{(q;q)_{k}}
	\bigg) \\
	&=
	\frac{q^{\ell+\alpha+1}}{1-q^{\alpha+1}}
	\bigg(
	\frac{(q^{m-\ell};q)_{\ell+1}}{1-q^{\ell+1}}
	-\frac{(q;q)_{\infty}}{(q^{\ell+\alpha+2};q)_{\infty}}
	\frac{(q^{\ell+\alpha+3};q)_{m-\ell-1}}{(q;q)_{m-\ell-1}}
	\bigg)\\
	&=
	\frac{q^{\ell+\alpha+1}}{1-q^{\alpha+1}}
	\bigg(
	\frac{(q^{m-\ell};q)_{\ell+1}}{1-q^{\ell+1}}
	-\frac
	{(q^{m-\ell};q)_{\ell+1}(q^{m+1};q)_{\infty}}
	{(1-q^{\ell+\alpha+2})(q^{m+\alpha+2};q)_{\infty}}
	\bigg).
\end{align*}
Finally we find that
\begin{equation}
	\label{eq:Tk-sum-finite-res}
	\sum_{k=0}^{m-1}
	(q^{k};q^{-1})_{\ell}
	\frac
	{\tilde{T}_{k}^{(\alpha)}}
	{\tilde{\lambda}_{k}^{(\alpha)}\tilde{\pi}_{k}^{(\alpha)}} 
	= 
	\frac
	{q^{\ell+\alpha+1}(q^{m-\ell};q)_{\ell+1}}
	{(1-q^{\alpha+1})(1-q^{\ell+\alpha+2})}
	\bigg(
	\frac
	{q^{\ell+1}(1-q^{\alpha+1})}
	{1-q^{\ell+1}}
	+ 1 - \gamma(m+1,\alpha+1;q)
	\bigg).
\end{equation}
Note that \eqref{eq:Tk-sum-finite-res} holds for $m \leq \ell$ as
well since both sides vanish.

Sending $m \to \infty$ in \eqref{eq:Tk-sum-finite-res} we get \eqref{eq:Tk-sum}.

Next, again assuming that $m \geq \ell+1$, we consider the sum
\begin{align*}
	\sum_{k=0}^{m-1}
	(q^{k};q^{-1})_{\ell}\,
	\frac
	{(\tilde{T}_{k}^{(\alpha)})^{2}}
	{\tilde{\lambda}_{k}^{(\alpha)}\tilde{\pi}_{k}^{(\alpha)}} 
	&= 
	\big(1-\gamma(1,\alpha+1;q)\big)
	\frac{q^{\alpha+1}}{1-q^{\alpha+1}}
	\sum_{k=0}^{m-1}
	q^{k} (q^{k};q^{-1})_{\ell}\,
	\frac{(q^{\alpha+2};q)_{k}}{(q;q)_{k}} \\
	&=
	\big(1-\gamma(1,\alpha+1;q)\big)
	\frac
	{q^{\ell+\alpha+1}(q^{\alpha+2};q)_{\ell}}
	{1-q^{\alpha+1}}
	\sum_{k=0}^{m-\ell-1}
	q^{k}\, \frac{(q^{\ell+\alpha+2};q)_{k}}{(q;q)_{k}}.
\end{align*}
Applying \eqref{eq:q-sup-k-q-beta-sum} we get the expression
\begin{align*}
	\big(1-\gamma(1,\alpha+1;q) \big)
	\frac
	{q^{\ell+\alpha+1} (q^{\alpha+2};q)_{\ell} (q^{\ell+\alpha+3};q)_{m-\ell-1}}
	{(1-q^{\alpha+1})(q;q)_{m-\ell-1}} \\
	=
	\big(1-\gamma(1,\alpha+1;q) \big)
	\frac
	{q^{\ell+\alpha+1}\,(q^{m};q^{-1})_{\ell+1}(q^{\alpha+2};q)_{m}}
	{(1-q^{\alpha+1}) (1-q^{\ell+\alpha+2})(q;q)_{m}}.
\end{align*}
With regard to \eqref{eq:Tk-Laguerre-mod} we obtain
\begin{equation}
	\label{eq:Tk-sq-sum-finite}
	\sum_{k=0}^{m-1}
	(q^{k};q^{-1})_{\ell}\,
	\frac
	{(\tilde{T}_{k}^{(\alpha)})^{2}}
	{\tilde{\lambda}_{k}^{(\alpha)}\tilde{\pi}_{k}^{(\alpha)}}
	=
	\frac
	{q^{\ell+\alpha+1}\,(q^{m};q^{-1})_{\ell+1}}
	{(1-q^{\alpha+1})(1-q^{\ell+\alpha+2})}
	\big(1 - \gamma(k+1,\alpha+1;q) \big) \tilde{T}_{m}^{(\alpha)}.
\end{equation}

A combination of \eqref{eq:Tk-sum-finite-res} and \eqref{eq:Tk-sq-sum-finite}
gives \eqref{eq:Tk-Tm-sum}. 
\end{proof}

\begin{proposition} 
The sequences $(\tilde{\lambda}_{k}^{(\alpha)} \pi_{k} )$
and $(\tilde{T}_{k}^{(\alpha)})$, as given in \eqref{eq:lbd-pi-Laguerre-mod}
and \eqref{eq:Tk-Laguerre-mod}, respectively, satisfy the equations,
for $m \geq 1$ and $\ell \geq 0$,
\begin{equation}
	\label{eq:Fchar-m-term}
	\sum_{0 \leq k_{1} < k_{2} < \ldots <k_{m}}
	(q^{k_{1}};q^{-1})_{\ell}\,
	\frac
	{\tilde{T}_{k_{1}}^{(\alpha)}
	(\tilde{T}_{k_{2}}^{(\alpha)} - \tilde{T}_{k_{1}}^{(\alpha)})
	\dots
	(\tilde{T}_{k_{m}}^{(\alpha)} - \tilde{T}_{k_{m-1}}^{(\alpha)})}
	{\tilde{\lambda}_{k_{1}}^{(\alpha)} \tilde{\pi}_{k_{1}}^{(\alpha)}
	\dots
	\tilde{\lambda}_{k_{m-1}}^{(\alpha)} \tilde{\pi}_{k_{m-1}}^{(\alpha)}
	\tilde{\lambda}_{k_{m}}^{(\alpha)} \tilde{\pi}_{k_{m}}^{(\alpha)}}
	=
	\frac
	{q^{m\,(2\ell+m+\alpha+1)}}
	{(q^{\ell+1};q)_{m}(q^{\ell+\alpha+2};q)_{m}}.
\end{equation}
\end{proposition}             
\begin{proof} We shall proceed by mathematical induction on $m$.
Equation \eqref{eq:Fchar-m-term} for $m=1$ coincides with \eqref{eq:Tk-sum}.
Suppose \eqref{eq:Fchar-m-term} is true for a given $m \in \NN$.
Then, in view of \eqref{eq:Tk-Tm-sum} and using the induction hypothesis,
\begin{align*}
	&\sum_{0 \leq k_{1} < k_{2} < \ldots < k_{m} < k_{m+1}}
	(q^{k_{1}};q^{-1})_{\ell}\,
	\frac
	{\tilde{T}_{k_{1}}^{(\alpha)}
	(\tilde{T}_{k_{2}}^{(\alpha)} - \tilde{T}_{k_{1}}^{(\alpha)})
	(\tilde{T}_{k_{3}}^{(\alpha)} - \tilde{T}_{k_{2}}^{(\alpha)})
	\dots
	(\tilde{T}_{k_{m+1}}^{(\alpha)} - \tilde{T}_{k_{m}}^{(\alpha)})}
	{\tilde{\lambda}_{k_{1}}^{(\alpha)} \tilde{\pi}_{k_{1}}^{(\alpha)}
	\tilde{\lambda}_{k_{2}}^{(\alpha)} \tilde{\pi}_{k_{2}}^{(\alpha)}
	\dots
	\tilde{\lambda}_{k_{m}}^{(\alpha)} \tilde{\pi}_{k_{m}}^{(\alpha)}
	\tilde{\lambda}_{k_{m+1}}^{(\alpha)} \tilde{\pi}_{k_{m+1}}^{(\alpha)}} \\
	&=
	\frac
	{q^{2\ell+\alpha+2}}
	{(1-q^{\ell+1}) (1-q^{\ell+\alpha+2})}\,
	\sum_{0 \leq k_{2} < k_{3} \ldots < k_{m} < k_{m+1}}
	(q^{k_{2}};q^{-1})_{\ell+1}\,
	\frac
	{\tilde{T}_{k_{2}}^{(\alpha)}
	(\tilde{T}_{k_{3}}^{(\alpha)} - \tilde{T}_{k_{2}}^{(\alpha)})
	\dots
	(\tilde{T}_{k_{m+1}}^{(\alpha)} - \tilde{T}_{k_{m}}^{(\alpha)})}
	{\tilde{\lambda}_{k_{2}}^{(\alpha)}\tilde{\pi}_{k_{2}}^{(\alpha)}
	\dots
	\tilde{\lambda}_{k_{m}}^{(\alpha)} \tilde{\pi}_{k_{m}}^{(\alpha)}
	\tilde{\lambda}_{k_{m+1}}^{(\alpha)} \tilde{\pi}_{k_{m+1}}^{(\alpha)}} \\
	&=
	\frac
	{q^{(m+1)\,(2\ell+m+\alpha+2)}}
	{(q^{\ell+1};q)_{m+1}(q^{\ell+\alpha+2};q)_{m+1}}.
\end{align*}
This shows the induction step. 
\end{proof}

\begin{proof}[Another proof of formula \eqref{eq:Fchar-alpha}] 
Letting $\ell = 0$ in \eqref{eq:Fchar-m-term} we obtain
\[
	\sum_{0 \leq k_{1} < k_{2} < \ldots < k_{m}}
	\frac
	{\tilde{T}_{k_{1}}^{(\alpha)}
	(\tilde{T}_{k_{2}}^{(\alpha)} - \tilde{T}_{k_{1}}^{(\alpha)})
	\dots
	(\tilde{T}_{k_{m}}^{(\alpha)} - \tilde{T}_{k_{m-1}}^{(\alpha)})}
	{\tilde{\lambda}_{k_{1}}^{(\alpha)} \tilde{\pi}_{k_{1}}^{(\alpha)}
	\dots
	\tilde{\lambda}_{k_{m-1}}^{(\alpha)} \tilde{\pi}_{k_{m-1}}^{(\alpha)}
	\tilde{\lambda}_{k_{m}}^{(\alpha)} \tilde{\pi}_{k_{m}}^{(\alpha)}}
	=
	\frac{q^{m\,(m+\alpha+1)}}{(q;q)_{m}(q^{\alpha+2};q)_{m}}.
\]
In view of \eqref{eq:Fchar-T}, the characteristic function equals
\[
	\tilde{\frakF}^{(\alpha)}(z)
	=
	1 + 
	\sum_{m=1}^{\infty}
	(-1)^{m} z^{m}
	\frac{q^{m\,(m+\alpha+1)}}{(q;q)_{m}(q^{\alpha+2};q)_{m}}
	=
	\rphis{0}{1}{-}{q^{\alpha+2}}{q,-q^{\alpha+2} z}. \qedhere
\]
\end{proof}

\begin{bibliography}{jacobi1}
	\bibliographystyle{amsplain}
\end{bibliography}
\end{document}